\documentclass[11pt, twoside, leqno]{article}

\usepackage{amssymb}
\usepackage{amsmath}
\usepackage{amsthm}
\usepackage{color}
\usepackage{mathrsfs}
\usepackage{txfonts}

\allowdisplaybreaks

\def\rr{{\mathbb R}}
\def\rn{{\mathbb{R}^n}}
\def\zz{{\mathbb Z}}
\def\cc{{\mathbb C}}
\def\nn{{\mathbb N}}
\def\cl{{\mathcal L}}
\def\cp{{\mathcal P}}
\def\cs{{\mathcal S}}
\def\cx{{\mathcal X}}

\def\fz{\infty }
\def\az{\alpha}
\def\bz{\beta}
\def\dz{\delta}
\def\lz{\lambda}

\def\lf{\left}
\def\r{\right}
\def\hs{\hspace{0.25cm}}
\def\ls{\lesssim}

\def\noz{\nonumber}
\def\wz{\widetilde}
\def\st{\subset}
\def\com{\complement}

\def\dis{\displaystyle}
\def\loc{{\mathop\mathrm{\,loc\,}}}
\def\supp{\mathop\mathrm{\,supp\,}}

\def\HL{M_{{\rm HL}}}
\def\q1{\wz q}

\def\vlpq{{L^{p(\cdot),q}(\rn)}}
\def\vh{{H_A^{p(\cdot)}(\rn)}}
\def\vhlpq{{H_A^{p(\cdot),q}(\rn)}}
\def\vahlpq{{H_A^{p(\cdot),r,s,q}(\rn)}}
\def\lfz{{L^{\fz}(\rn)}}
\def\lv{{L^{p(\cdot)}(\rn)}}
\def\Bik{{B_i^k}}
\def\Qik{{Q_i^k}}
\def\lik{{\lz_i^k}}
\def\aik{{a_i^k}}
\def\xik{{x_i^k}}

\newtheorem{theorem}{Theorem}[section]
\newtheorem{lemma}[theorem]{Lemma}
\newtheorem{corollary}[theorem]{Corollary}
\newtheorem{proposition}[theorem]{Proposition}

\theoremstyle{definition}
\newtheorem{remark}[theorem]{Remark}
\newtheorem{definition}[theorem]{Definition}
\renewcommand{\appendix}{\par
   \setcounter{section}{0}%
   \setcounter{subsection}{0}%
   \setcounter{subsubsection}{0}%
   \gdef\thesection{\@Alph\c@section}%
   \gdef\thesubsection{\@Alph\c@section.\@arabic\c@subsection}%
   \gdef\theHsection{\@Alph\c@section.}%
   \gdef\theHsubsection{\@Alph\c@section.\@arabic\c@subsection}%
   \csname appendixmore\endcsname
 }

\numberwithin{equation}{section}

\begin{document}

\arraycolsep=1pt

\title{\bf\Large Anisotropic Variable Hardy-Lorentz Spaces
and Their Real Interpolation
\footnotetext{\hspace{-0.35cm} 2010 {\it
Mathematics Subject Classification}. Primary  42B35;
Secondary 46E30, 42B30, 42B25, 46B70.
\endgraf {\it Key words and phrases.} variable exponent,
(Hardy-)Lorentz space, expansive matrix, atom,
Lusin area function, real interpolation.
\endgraf This project is supported by the National
Natural Science Foundation of China
(Grant Nos.~11571039, 11671185 and 11471042).}}
\author{Jun Liu, Dachun Yang\footnote{Corresponding author / April 26, 2017.}\ \
and Wen Yuan}
\date{}
\maketitle

\vspace{-0.8cm}

\begin{center}
\begin{minipage}{13cm}
{\small {\bf Abstract}\quad
Let $p(\cdot):\ \mathbb R^n\to(0,\infty)$ be a variable exponent
function satisfying the globally log-H\"{o}lder continuous condition,
$q\in(0,\infty]$ and $A$ be a general expansive matrix on $\mathbb{R}^n$.
In this article, the authors first introduce the anisotropic variable
Hardy-Lorentz space $H_A^{p(\cdot),q}(\mathbb R^n)$ associated
with $A$, via the radial grand maximal function, and then establish its
radial or non-tangential maximal function characterizations. Moreover,
the authors also obtain characterizations of
$H_A^{p(\cdot),q}(\mathbb R^n)$, respectively, in terms of the atom
and the Lusin area function. As an application, the authors
prove that the anisotropic variable Hardy-Lorentz space
$H_A^{p(\cdot),q}(\mathbb R^n)$ severs as the intermediate space
between the anisotropic variable Hardy space $H_A^{p(\cdot)}(\mathbb R^n)$
and the space $L^\infty(\mathbb R^n)$ via the real interpolation.
This, together with a special case of the real interpolation theorem of H. Kempka and J. Vyb\'iral
on the variable Lorentz space, further implies the coincidence
between $H_A^{p(\cdot),q}(\mathbb R^n)$ and the variable Lorentz space $L^{p(\cdot),q}(\mathbb R^n)$
when $\mathop\mathrm{essinf}_{x\in\mathbb{R}^n}p(x)\in (1,\infty)$.

}
\end{minipage}
\end{center}

\section{Introduction\label{s1}}
\hskip\parindent
As a generalization of the classical Lebesgue spaces $L^p(\rn)$,
the variable Lebesgue spaces $\lv$, in which the constant
exponent $p$ is replaced by an exponent function
$p(\cdot):\ \rn\to(0,\fz)$, were studied by Musielak \cite{m83}
and Nakano \cite{n50,n51}, which can be traced back to Orlicz \cite{o31,o32}.
But the modern theory of function spaces with variable exponents was started with the articles \cite{kr91} of
Kov\'a\v{c}ik and R\'akosn\'{\i}k  and \cite{fz01} of Fan and Zhao
as well as \cite{cruz03} of Cruz-Uribe and \cite{din04} of Diening,
and nowadays has been widely used in harmonic analysis
(see, for example, \cite{cf13,dhhr11,ylk17}). In addition, the theory of
variable function spaces also has interesting applications
in fluid dynamics \cite{am02}, image processing \cite{clr06},
partial differential equations and variational calculus \cite{am05,f07,hhv08,su09}.

Recently, Nakai and Sawano \cite{ns12} and, independently,
Cruz-Uribe and Wang \cite{cw14} with
some weaker assumptions on $p(\cdot)$ than those used in \cite{ns12},
extended the theory of variable Lebesgue spaces via investigating
the variable Hardy spaces on $\rn$. Later, Sawano \cite{s10},
Zhuo et al. \cite{zyl16} and Yang et al. \cite{yzn16} further completed
the theory of these variable Hardy spaces. For more developments of
function spaces with variable exponents, we refer the reader to
\cite{ah10,dhr09,kv14,n16,ns2012,v09,xu08,xu2008,yyyz16,yzy15}
and their references. In particular, Kempka and Vyb\'iral \cite{kv14}
introduced the variable Lorentz spaces which were a generalization of both
the variable Lebesgue spaces and the classical Lorentz spaces and obtained some
basic properties of these spaces including several embedding conclusions.
The real interpolation result that the variable Lorentz space serves as the
intermediate space between the variable Lebesgue space $\lv$ and the space
$\lfz$ was also presented in \cite{kv14}.

Very recently, Yan et al. \cite{yyyz16}
first introduced the variable weak Hardy spaces on $\rn$ and established
various real-variable characterizations of these spaces; as application,
the boundedness of some Calder\'{o}n-Zygmund operators in the critical
case was also presented. Based on these results, via establishing a very
interesting decomposition for any distribution of the variable weak Hardy space,
Zhuo et al. \cite{zyy17} proved the following real interpolation theorem between
the variable Hardy space $H^{p(\cdot)}(\mathbb R^n)$
and the space $L^{\infty}(\mathbb R^n)$:
\begin{equation}\label{e1.1}
(H^{p(\cdot)}(\mathbb R^n),L^{\infty}(\mathbb R^n))_{\theta,\infty}=
W\!H^{p(\cdot)/(1-\theta)}(\mathbb R^n),\quad \theta\in(0,1),
\end{equation}
where $W\!H^{p(\cdot)/(1-\theta)}(\mathbb R^n)$ denotes the variable
weak Hardy space and $(\cdot,\cdot)_{\theta,\fz}$ the real interpolation.

As was well known, Fefferman et al. \cite{frs74} showed that the Hardy-Lorentz
space $H^{p,q}(\rn)$ was actually the intermediate space between the classical Hardy
space $H^{p}(\rn)$ and the space $\lfz$ under the real interpolation, which is
the main motivation to develop the real-variable theory of $H^{p,q}(\rn)$.
Thus, it is natural and interesting  to ask whether or not the variable Hardy-Lorentz
space also serves as the intermediate space between the variable Hardy space
$H^{p(\cdot)}(\rn)$ and the space $\lfz$ via the real interpolation, namely,
if $(\cdot,\cdot)_{\theta,\fz}$ in \eqref{e1.1} is replaced by $(\cdot,\cdot)_{\theta,q}$
with $q\in (0,\fz]$, what happens?

On the other hand, as the series of works
(see, for example, \cite{at07,ac10, a94, frs74,fs87,l91,p05}) reveal,
the Hardy-Lorentz spaces (as well the weak Hardy spaces)
serve as a more subtle research object than the usual Hardy
spaces when studying the boundedness of singular integrals, especially,
in some critical cases, due to the fact that these function spaces own
finer structures. Moreover, after the celebrated
articles \cite{c77,ct75,ct77} of Calder\'{o}n and Torchinsky on parabolic
Hardy spaces, there has been an enormous interest in extending classical
function spaces arising in harmonic analysis from Euclidean spaces to
some more general underlying spaces; see, for example,
\cite{ds16,fs82,hmy06,hmy08,st87,s13,s16,t83,t92,yyh13}. The function spaces in the
anisotropic setting have proved of wide generality (see, for example,
\cite{mb03,mb07,bh06}), which include the classical isotropic spaces and the parabolic spaces as
special cases. For more progresses about this theory, we refer the reader to
\cite{lby14,lbyy12,lyy16,lyy16LP,lyy17,fhly17,t06,t15}
and their references. In particular, the authors recently introduced
the anisotropic Hardy-Lorentz spaces $H_A^{p,q}(\rn)$,
associated with some dilation $A$, and obtained their various real-variable
characterizations (see \cite{lyy16,lyy16LP}). Also,
very recently, Zhuo et al. \cite{zsy16} developed the real-variable theory
of the variable Hardy space $H^{p(\cdot)}(\cx)$ on an RD-space $\cx$.
Recall that a metric measure space of homogeneous type $\cx$ is called an RD-space if
it is a metric measure space of homogeneous type in the sense of Coifman and Weiss \cite{cw71,cw77}
and satisfies some reverse doubling property, which was originally introduced
by Han et al. \cite{hmy08} (see also \cite{zy11} for some equivalent
characterizations).

To further study the intermediate space between the variable Hardy space
$H^{p(\cdot)}(\rn)$ and the space $\lfz$ via the real interpolation and
also to give a complete theory of variable Hardy-Lorentz spaces in anisotropic
setting, in this article, we first introduce the anisotropic variable Hardy-Lorentz
space, via the radial grand maximal function, and then establish its several
real-variable characterizations, respectively,
in terms of the atom, the radial or the non-tangential
maximal functions, and the Lusin area function. As an application, we prove
that the anisotropic variable Hardy-Lorentz space $\vhlpq$ severs as the
intermediate space between the anisotropic variable Hardy space
$H_A^{p(\cdot)}(\rn)$ and the space $\lfz$ via the real interpolation.
This, together with a special case of the real interpolation theorem of
Kempka and Vyb\'iral in \cite{kv14} on the variable Lorentz space, further implies the coincidence
between $H_A^{p(\cdot),q}(\mathbb R^n)$ and the variable Lorentz space $L^{p(\cdot),q}(\mathbb R^n)$
when $\mathop\mathrm{essinf}_{x\in\mathbb{R}^n}p(x)\in (1,\infty)$.

To be precise, this article is organized as follows.

In Section \ref{s2}, we first recall some notation and notions on
Euclidean spaces, with anisotropic dilations, and variable
Lebesgue spaces as well as some basic properties of these spaces to be
used in this article. Then we introduce the anisotropic variable
Hardy-Lorentz space $\vhlpq$ via the radial grand maximal function.

Section \ref{s3} is aimed to characterize $\vhlpq$ by means of
the radial or the non-tangential maximal functions (see Theorem \ref{tt1} below).
To this end, via the Aoki-Rolewicz theorem (see \cite{ta42, sr57}),
we first prove that the $\vlpq$ quasi-norm of the tangential maximal
function $T^{N(K,L)}_\phi(f)$ can be controlled by that of the
non-tangential maximal function $M^{(K,L)}_\phi(f)$ for all
$f\in\cs'(\rn)$ (see Lemma \ref{tl1} below), where $K$ is the
truncation level, $L$ is the decay level and $\cs'(\rn)$ denotes
the set of all tempered distributions on $\rn$. Then, by the
boundedness of the Hardy-Littlewood maximal function as in \eqref{se5} below
on $\lv$ (see Lemma \ref{sl1} below) with
$p(\cdot)$ satisfying the so-called globally log-H\"older continuous
condition (see \eqref{se6} and \eqref{se7} below) and
$1<p_-\le p_+<\fz$, where $p_-$ and $p_+$ are as in \eqref{se3} below,
we obtain the boundedness of the Hardy-Littlewood maximal function
on $\vlpq$ (see Lemma \ref{sl7} below) with $p(\cdot)$ satisfying
the same condition as that in Lemma \ref{sl1} and $q\in(0,\fz]$.
We point out that the monotone convergence property for increasing
sequences on $\vlpq$ (see Proposition \ref{sl6} below) as well as Lemmas
\ref{sl1} and \ref{sl7} play a key role in proving Theorem \ref{tt1}.

In Section \ref{s4}, via borrowing some ideas from
\cite[Theorem 3.6]{lyy16} and \cite[Theorem 4.4]{yyyz16},
we establish the atomic characterization of $\vhlpq$. Indeed, we first introduce
the anisotropic variable atomic Hardy-Lorentz space $\vahlpq$ in
Definition \ref{fd2} below and then prove
$$\vhlpq=\vahlpq$$
with equivalent quasi-norms (see Theorem \ref{ft1} below). To prove
that $\vahlpq$ is continuously embedded into $\vhlpq$, motivated by
\cite[Lemma 4.1]{s10}, we first conclude that some estimates related to
$L^{p(\cdot)}(\rn)$ norms for some series of functions can be reduced
into dealing with the $L^r(\rn)$ norms of the corresponding functions
(see Lemma \ref{fl2} below), which actually is an anisotropic version of
\cite[Lemma 4.1]{s10}. Then, by using this key lemma and the Fefferman-Stein
vector-valued inequality of the Hardy-Littlewood maximal operator $\HL$ on
$\lv$ (see Lemma \ref{fl1} below), we prove that $\vahlpq\st\vhlpq$ and
the inclusion is continuous.
The method used in the proof for the converse embedding is different from
that used in the proof for the corresponding embedding of variable Hardy spaces
$H^{p(\cdot)}(\rn)$ (or, resp., anisotropic Hardy spaces $H_A^p(\rn)$).
Recall that $L_\loc^1(\rn)\cap H^{p(\cdot)}(\rn)$
(or, resp., $L_\loc^1(\rn)\cap H_A^p(\rn)$) is dense in $H^{p(\cdot)}(\rn)$
(or, resp., $H_A^p(\rn)$), which plays a key role in the atomic decomposition
of $H^{p(\cdot)}(\rn)$ (or, resp., $H_A^p(\rn)$). However, this standard procedure
is invalid for the space $H_A^{p(\cdot),\fz}(\rn)$, due to its lack of a dense function subspace.
To overcome this difficulty, we borrow some ideals from \cite[Theorem 3.6]{lyy16}
(see also \cite{dl08}), in which the authors directly obtained an atomic decomposition
for convolutions of distributions in $H_A^{p(\cdot),\fz}(\rn)$ and Schwartz functions instead
of some dense function subspace.

As an application of the atomic characterization of $\vhlpq$ obtained in
Theorem \ref{ft1}, in Section \ref{s5}, we establish the Lusin area function
characterization of $\vhlpq$ (see Theorem \ref{fivet1} below).
In the proof of Theorem \ref{fivet1}, the anisotropic Calder\'{o}n reproducing
formula and the method used in the proof of the atomic characterization
of $\vhlpq$ play a key role. However, when we decompose a distribution
into a sum of atoms, the dual method used in estimating the norm of each atom in the
classic case does not work anymore in the present setting. Instead,
a strategy, used in \cite{lyy16LP}, originated from Fefferman \cite{f88}, that
obtains a subtle estimate (see, for example, \cite[(3.23)]{lyy16LP}) plays a
key role here; see the estimate \eqref{five11} below.

In Section \ref{s6}, as another application of the atomic characterization of
$\vhlpq$, we prove the following real interpolation result between the
anisotropic variable Hardy space $H_A^{p(\cdot)}(\rn)$ and the space $\lfz$:
\begin{equation}\label{e1.2}
(H_A^{p(\cdot)}(\mathbb R^n),L^{\infty}(\mathbb R^n))_{\theta,q}=
H_A^{p(\cdot)/(1-\theta),q}(\mathbb R^n),\quad \theta\in(0,1),\ q\in(0,\fz]
\end{equation}
(see Theorem \ref{sixt1} below). To prove this result, via borrowing some ideas from
\cite{zyy17}, we first obtain a decomposition for any distribution of the anisotropic
variable Hardy-Lorentz space into ``good" and ``bad" parts (see Lemma \ref{sixl1}
below), which is of independent interest. We point out that, as a special case of
\cite[Theorem 4.3(i)]{zsy16}, we know that the atomic characterization of
$\vh$ holds true. This, together with the vector-valued inequality of the
Hardy-Littlewood maximal function on the variable Lebesgue space
(see Lemma \ref{fl1} below), plays a key role in the proof of Lemma \ref{sixl1}.
Applying \eqref{e1.2}, together with \cite[Corollary 4.20]{zsy16} on the coincidence
between $H_A^{p(\cdot)}(\mathbb R^n)$ and $L^{p(\cdot)}(\mathbb R^n)$
as well as \cite[Remark 4.2(ii)]{kv14}, we further obtain the coincidence
between $H_A^{p(\cdot),q}(\mathbb R^n)$ and $L^{p(\cdot),q}(\mathbb R^n)$
when $\mathop\mathrm{essinf}_{x\in\mathbb{R}^n}p(x)\in (1,\infty)$;
see Corollary \ref{sixc1} below.

We should point out that, if $A:=d\,{\rm I}_{n\times n}$ for some
$d\in\rr$ with $|d|\in(1,\fz)$, here and hereafter,
${\rm I}_{n\times n}$ denotes the $n\times n$
\emph{unit matrix} and $|\cdot|$ denotes the
Euclidean norm in $\rn$, then the space $\vhlpq$ becomes the classical
isotropic variable Hardy-Lorentz space. In this case, the results in this article
are also independently obtained by Jiao et al. \cite{jzz17} via some slight
different methods.

Finally, we make some conventions on notation.
Throughout this article, we always let
$\mathbb{N}:=\{1,2,\ldots\}$ and
$\mathbb{Z}_+:=\{0\}\cup\mathbb{N}$.
For any multi-index
$\beta:=(\beta_1,\ldots,\beta_n)\in\mathbb{Z}_+^n$,
let
$|\beta|:=\beta_1+\cdots+\beta_n$.
We denote by $C$ a \emph{positive constant}
which is independent of the main parameters,
but its value may change from line to line. Moreover,
we use $f\ls g$ to denote $f\leq Cg$ and, if $f\ls g\ls f$,
we then write $f\sim g$. For any $q\in[1,\fz]$, we denote by $q'$
its \emph{conjugate index}, namely, $1/q+1/q'=1$. In addition,
for any set $E\subset\rn$, we denote by $E^\complement$ the
set $\rn\setminus E$, by $\chi_E$ its \emph{characteristic function}
and by $\sharp E$ the \emph{cardinality} of $E$. The symbol $\lfloor s\rfloor$,
for any $s\in\mathbb{R}$, denotes the \emph{largest integer not greater than $s$}.

\section{Preliminaries \label{s2}}
\hskip\parindent
In this section, we introduce the anisotropic variable Hardy-Lorentz space
via the radial grand maximal function. To this end, we first recall some
notation and notions on spaces of homogeneous type associated with dilations
and variable Lebesgue spaces as well as some basic conclusions of these spaces
to be used in this article. For an exposition of these concepts,
we refer the reader to the monographs
\cite{mb03,cf13,dhhr11}.

We begin with recalling the notion of expansive matrices
in \cite{mb03}.

\begin{definition}\label{d-em}
A real $n\times n$ matrix $A$ is called an \emph{expansive matrix}
(shortly, a \emph{dilation}) if
$$\min_{\lz\in\sigma(A)}|\lz|>1,$$
here and hereafter, $\sigma(A)$ denotes the \emph{collection of
all eigenvalues of $A$}.
\end{definition}

Throughout this article,
$A$ always denotes a fixed dilation
and $b:=|\det A|$. Then we easily find that
$b\in(1,\fz)$ by \cite[p.\,6, (2.7)]{mb03}.
Let $\lambda_-$ and $\lambda_+$
be two \emph{positive numbers} satisfying that
$$1<\lambda_-<\min\{|\lambda|:\
\lambda\in\sigma(A)\}
\leq\max\{|\lambda|:\
\lambda\in\sigma(A)\}<\lambda_+.$$
In the case when $A$ is diagonalizable over
$\mathbb{C}$, we can even take
$\lambda_-:=\min\{|\lambda|:\
\lambda\in\sigma(A)\}$
and
$\lambda_+:=\max\{|\lambda|:\
\lambda\in\sigma(A)\}$.
Otherwise, we need to choose them
sufficiently close to these equalities
according to what we need in the arguments below.

It was proved in \cite[p.\,5, Lemma 2.2]{mb03} that,
for a given dilation $A$, there exist an open
ellipsoid $\Delta$ and $r\in(1,\infty)$ such that
$\Delta\subset r\Delta\subset A\Delta$, and one
may additionally assume that $|\Delta|=1$,
where $|\Delta|$ denotes the \emph{n-}dimensional
Lebesgue measure of the set $\Delta$.
For any $k\in\zz$, let $B_k:=A^k\Delta$.
Obviously, $B_k$ is open,
$B_k\subset rB_k\subset B_{k+1}$ and $|B_k|=b^k$.
An ellipsoid $x+B_k$ for some $x\in\rn$ and $k\in\mathbb{Z}$
is called a \emph{dilated ball}.
Denote by $\mathfrak{B}$ the set of all such
dilated balls, namely,
\begin{align}\label{se14}
\mathfrak{B}:=\lf\{x+B_k:\ x\in\rn,\ k\in\mathbb{Z}\r\}.
\end{align}
Throughout this article, let $\tau$ be
the \emph{minimal integer} such that
$r^\tau\geq2$. Then, for any $k\in\mathbb{Z}$,
it holds true that
\begin{align}\label{se1}
B_k+B_k\subset B_{k+\tau}
\end{align}
and
\begin{align}\label{se2}
B_k+(B_{k+\tau})^\complement
\subset (B_k)^\complement,
\end{align}
where $E+F$ denotes the algebraic sum
$\{x+y:\ x\in E,\,y\in F\}$ of sets
$E,\,F\subset\mathbb{R^\emph{n}}$.

The notion of the homogeneous quasi-norm
induced by $A$ was introduced in
\cite[p.\,6, Definition 2.3]{mb03} as follows.

\begin{definition}\label{sd2}
A measurable mapping
$\rho:\ \rn \to [0,\infty)$
is called a \emph{homogeneous quasi-norm},
associated with a dilation $A$, if
\begin{enumerate}
\item[\rm{(i)}] $x\neq\vec0_n$ implies that $\rho(x)\in(0,\fz)$,
here and hereafter, $\vec0_n$ denotes the \emph{origin} of $\rn$;

\item[\rm{(ii)}] $\rho(Ax)=b\rho(x)$ for any $x\in\rn$;

\item[\rm{(iii)}] $\rho(x+y)\leq H[\rho(x)+\rho(y)]$
for any $x,\,y\in\rn$, where $H\in[1,\fz)$ is a constant
independent of $x$ and $y$.
\end{enumerate}
\end{definition}

In the standard dyadic case
$A:=2{\rm I}_{n\times n}$, $\rho(x):=|x|^n$
for any $x\in\rn$ is an example of the
homogeneous quasi-norm associated
with $A$. In \cite[p.\,6, Lemma 2.4]{mb03},
it was proved that all homogeneous
quasi-norms associated with $A$ are equivalent.
Therefore, for a given dilation $A$, in what follows,
we always use the
\emph{step homogeneous quasi-norm} $\rho$
defined by setting, for any $x\in\rn$,
\begin{equation*}
\rho(x):=\sum_{j\in\mathbb{Z}}
b^j\chi_{B_{j+1}\setminus B_j}(x)\hspace{0.25cm}
{\rm if}\ x\neq\vec0_n,\hspace{0.35cm} {\rm or\ else}
\hspace{0.25cm}\rho(\vec0_n):=0
\end{equation*}
for convenience.
Obviously, for any $k\in\zz$,
$B_k=\{x\in\rn:\ \rho(x)<b^k\}$.
Observe that $(\rn,\,\rho,\,dx)$ is a space
of homogeneous type in the sense of
Coifman and Weiss \cite{cw71,cw77}, here and hereafter,
$dx$ denotes the \emph{$n$-dimensional Lebesgue measure},
and, moreover, $(\rn,\,\rho,\,dx)$ is indeed an RD-space (see \cite{hmy08,zy11}).

Recall that a measurable function $p(\cdot):\ \rn\to(0,\fz)$ is called a
\emph{variable exponent}. For any variable exponent
$p(\cdot)$, let
\begin{align}\label{se3}
p_-:=\mathop\mathrm{ess\,inf}_{x\in \rn}p(x)\hspace{0.35cm}
{\rm and}\hspace{0.35cm}
p_+:=\mathop\mathrm{ess\,sup}_{x\in \rn}p(x).
\end{align}
Denote by $\cp(\rn)$ the \emph{set of all variable exponents}
$p(\cdot)$ satisfying $0<p_-\le p_+<\fz$.

Let $f$ be a measurable function on $\rn$ and $p(\cdot)\in\cp(\rn)$.
Then the \emph{modular functional} (or, for simplicity, the \emph{modular})
$\varrho_{p(\cdot)}$, associated with $p(\cdot)$, is defined by setting
$$\varrho_{p(\cdot)}(f):=\int_\rn|f(x)|^{p(x)}\,dx$$ and the
\emph{Luxemburg} (also called \emph{Luxemburg-Nakano})
\emph{quasi-norm} $\|f\|_{\lv}$ by
\begin{equation*}
\|f\|_{\lv}:=\inf\lf\{\lz\in(0,\fz):\ \varrho_{p(\cdot)}(f/\lz)\le1\r\}.
\end{equation*}
Moreover, the \emph{variable Lebesgue space} $\lv$ is defined to be the
set of all measurable functions $f$ satisfying that $\varrho_{p(\cdot)}(f)<\fz$,
equipped with the quasi-norm $\|f\|_{\lv}$.

\begin{remark}\label{sr1}
 Let $p(\cdot)\in\cp(\rn)$.
\begin{enumerate}
\item[(i)] Obviously, for any $r\in (0,\fz)$ and $f\in\lv$,
$$\lf\||f|^r\r\|_{\lv}=\|f\|_{L^{rp(\cdot)}(\rn)}^r.$$
Moreover, for any $\mu\in{\mathbb C}$ and $f,\ g\in\lv$,
$\|\mu f\|_{\lv}=|\mu|\|f\|_{\lv}$ and
$$\|f+g\|_{\lv}^{\underline{p}}\le \|f\|_{\lv}^{\underline{p}}
+\|g\|_{\lv}^{\underline{p}},$$
here and hereafter,
\begin{align}\label{se4}
\underline{p}:=\min\{p_-,1\}
\end{align}
with $p_-$ as in \eqref{se3}.
In particular, when $p_-\in[1,\fz]$,
$\lv$ is a Banach space (see \cite[Theorem 3.2.7]{dhhr11}).

\item[(ii)] It was proved in \cite[Proposition 2.21]{cf13} that,
for any function $f\in\lv$ with $\|f\|_{\lv}>0$,
$\varrho_{p(\cdot)}(f/\|f\|_{\lv})=1$
and, in \cite[Corollary 2.22]{cf13} that, if $\|f\|_{\lv}\le1$,
then $\varrho_{p(\cdot)}(f)\le\|f\|_{\lv}$.
\end{enumerate}
\end{remark}

A function $p(\cdot)\in\cp(\rn)$ is said to satisfy the
\emph{globally log-H\"older continuous condition}, denoted by $p(\cdot)\in C^{\log}(\rn)$,
if there exist two positive constants $C_{\log}(p)$ and $C_\fz$, and
$p_\fz\in\rr$ such that, for any $x,\ y\in\rn$,
\begin{equation}\label{se6}
|p(x)-p(y)|\le \frac{C_{\log}(p)}{\log(e+1/\rho(x-y))}
\end{equation}
and
\begin{equation}\label{se7}
|p(x)-p_\fz|\le \frac{C_\fz}{\log(e+\rho(x))}.
\end{equation}

The following variable Lorentz space $\vlpq$
is known as a special case of the variable Lorentz space
$L^{p(\cdot),q(\cdot)}(\rn)$ investigated by Kempka and
Vyb\'iral in \cite{kv14}.

\begin{definition}\label{sd4}
Let $p(\cdot)\in\cp(\rn)$.
The \emph{variable Lorentz space} $\vlpq$
is defined to be the set of all measurable functions $f$ such that
\begin{align*}
\|f\|_{\vlpq}:=\left\{
\begin{array}{ll}
\lf[\dis\int_0^\fz
\lz^q\lf\|\chi_{\{x\in\rn:\ |f(x)|>\lz\}}\r\|_{\lv}^q\,\frac{d\lz}{\lz}\r]^{\frac{1}{q}}
\hspace{0.5cm} &{\rm when}\hspace{0.5cm} q\in(0,\fz),\\
\dis\sup_{\lz\in(0,\fz)}\lf[\lz\lf\|\chi_{\{x\in\rn:\ |f(x)|>\lz\}}\r\|_{\lv}\r]
\hspace{1.15cm}&{\rm when}\hspace{0.5cm} q=\infty
\end{array}\r.
\end{align*}
is finite.
\end{definition}

From \cite[Lemma 2.4 and Theorem 3.1]{kv14},
we deduce the following Lemmas \ref{sl3*} and \ref{sl3}, respectively.

\begin{lemma}\label{sl3*}
Let $p(\cdot)\in \cp(\rn)$ and $q\in(0,\fz]$. Then, for
any measurable function $f$,
$$\|f\|_{\vlpq}\sim\lf[\sum_{k\in\zz}2^{kq}
\lf\|\chi_{\{x\in\rn:\ |f(x)|>2^k\}}\r\|_{\lv}^q\r]^{1/q}$$
with the usual interpretation for $q=\fz$, where the equivalent
positive constants are independent of $f$.
\end{lemma}

\begin{lemma}\label{sl3}
Let $p(\cdot)\in \cp(\rn)$ and $q\in(0,\fz]$. Then $\|\cdot\|_{\vlpq}$
defines a quasi-norm on $\vlpq$.
\end{lemma}

It is easy to obtain the following result, the details
being omitted.
\begin{lemma}\label{sl4}
Let $p(\cdot)\in \cp(\rn)$ and $q\in(0,\fz]$.
Then, for any $r\in(0,\fz)$ and $f\in\vlpq$,
$$\lf\||f|^r\r\|_{\vlpq}=\|f\|_{L^{rp(\cdot),rq}(\rn)}^r.$$
\end{lemma}

From the monotone convergence theorem of $L^{p(\cdot)}(\rn)$
(see \cite[Corollary 2.64]{cf13}), we easily deduce the following
monotone convergence property of $\vlpq$, the details being omitted.

\begin{proposition}\label{sl6}
Let $p(\cdot)\in\cp(\rn)$ and $\{f_k\}_{k\in\nn}\subset\vlpq$
be some sequence of non-negative functions satisfying that
$f_k$, as $k\to\fz$, increases pointwisely almost everywhere
to $f\in\vlpq$ in $\rn$. Then
$$\|f-f_k\|_{\vlpq}\to0\hspace{0.5cm}{\rm as}\ \ k\to\fz.$$
\end{proposition}

Throughout this paper, denote by $\cs(\rn)$ the space of
all Schwartz functions, namely, the set of all $C^\infty(\rn)$
functions $\varphi$ satisfying that, for every integer $\ell\in\zz_+$ and
multi-index $\alpha\in(\zz_+)^n$,
$$\|\varphi\|_{\alpha,\ell}:=
\sup_{x\in\rn}[\rho(x)]^\ell
\lf|\partial^\alpha\varphi(x)\r|<\infty.$$
These quasi-norms $\{\|\cdot\|_{\az,\ell}\}_{\az\in(\zz_+)^n,\ell\in\zz_+}$
also determine the topology of $\cs(\rn)$.
We use $\cs'(\rn)$ to denote the dual space of $\cs(\rn)$, namely,
the space of all tempered distributions on $\rn$ equipped with the
weak-$\ast$ topology. For any $N\in\mathbb{Z}_+$, let
$$\cs_N(\rn):=\{\varphi\in\cs(\rn):\
\|\varphi\|_{\alpha,\ell}\leq1,\
|\alpha|\leq N,\ \ell\leq N\};$$
equivalently,
\begin{align*}
\varphi\in\cs_N(\rn)\Longleftrightarrow
\|\varphi\|_{\cs_N(\rn)}:=\sup_{|\alpha|\leq N}
\sup_{x\in\rn}\lf[\lf|\partial^\alpha
\varphi(x)\r|\max\lf\{1,\lf[
\rho(x)\r]^N\r\}\r]\leq1.
\end{align*}
In what follows, for any
$\varphi\in\cs(\rn)$ and $k\in\mathbb{Z}$, let
$\varphi_k(\cdot):=b^{-k}\varphi(A^{-k}\cdot)$.

\begin{definition}\label{sd1}
Let $\varphi\in\cs(\rn)$ and $f\in\cs'(\rn)$. The
\emph{non-tangential maximal function} $M_\varphi(f)$
and the \emph{radial maximal function} $M_\varphi^0(f)$
of $f$ with respect to $\varphi$ are defined, respectively,
by setting, for any $x\in\rn$,
\begin{align}\label{se11}
M_\varphi(f)(x):= \sup_{y\in x+B_k,
k\in\mathbb{Z}}|f\ast\varphi_k(y)|
\end{align}
and
\begin{equation}\label{se10}
M_\varphi^0(f)(x):= \sup_{k\in\mathbb{Z}}
|f\ast\varphi_k(x)|.
\end{equation}
For any given $N\in\mathbb{N}$, the
\emph{non-tangential grand maximal function} $M_N(f)$
and the \emph{radial grand maximal function} $M_N^0(f)$
of $f\in\cs'(\rn)$ are defined, respectively, by setting,
for any $x\in\rn$,
\begin{equation}\label{se8}
M_N(f)(x):=\sup_{\varphi\in\cs_N(\rn)}
M_\varphi(f)(x)
\end{equation}
and
\begin{equation*}
M_N^0(f)(x):=\sup_{\varphi\in\cs_N(\rn)}
M_\varphi^0(f)(x).
\end{equation*}
\end{definition}

We now introduce anisotropic variable Hardy-Lorentz spaces as follows.

\begin{definition}\label{sd3}
Let $p(\cdot)\in C^{\log}(\rn)$, $q\in(0,\fz]$ and
$N\in\mathbb{N}\cap[\lfloor(\frac1{\underline{p}}-1)\frac{\ln b}{\ln
\lambda_-}\rfloor+2,\fz)$,
where $\underline{p}$ is as in \eqref{se4}.
The \emph{anisotropic variable Hardy-Lorentz space},
denoted by $\vhlpq$, is defined by setting
\begin{equation*}
\vhlpq
:=\lf\{f\in\cs'(\rn):\ M_N^0(f)\in\vlpq\r\}
\end{equation*}
and, for any $f\in\vhlpq$, let
$\|f\|_{\vhlpq}
:=\| M_N^0(f)\|_{\vlpq}$.
\end{definition}

\begin{remark}\label{sr2}
\begin{enumerate}
\item[(i)] Even though the quasi-norm of $\vhlpq$ in Definition \ref{sd3}
depends on $N$, it follows from
Theorem \ref{tt1} below that the space $\vhlpq$
is independent of the choice of $N$ as long as
$N\in[\lfloor(\frac1{\underline{p}}-1)\frac{\ln b}{\ln
\lambda_-}\rfloor+2,\fz)$. If $p(\cdot)\equiv p\in(0,\fz)$,
then the space $\vhlpq$ is just the anisotropic Hardy-Lorentz space
$H^{p,q}_A(\rn)$ investigated by Liu et al. in \cite{lyy16} and, if
$A:=d\,{\rm I}_{n\times n}$ for some $d\in\rr$ with $|d|\in(1,\fz)$ and $q=\fz$,
then the space $H^{p(\cdot),\fz}_A(\rn)$ becomes the variable weak Hardy space introduced
by Yan et al. in \cite{yyyz16}.

\item[(ii)] Very recently, via the variable Lorentz spaces
$\cl^{p(\cdot),q(\cdot)}(\rn)$ in \cite{eks08}, where
$$p(\cdot),\ q(\cdot):\ (0,\fz)\to(0,\fz)$$
are bounded measurable functions, Almeida et al. \cite{abr16}
investigated the anisotropic variable Hardy-Lorentz spaces
$H^{p(\cdot),q(\cdot)}(\rn, A)$ on $\rn$.
As was mentioned in
\cite[Remark 2.6]{kv14}, the space $\cl^{p(\cdot),q(\cdot)}(\rn)$
in \cite{eks08} never goes back to the space $\lv$,
since the variable exponent $p(\cdot)$ in $\cl^{p(\cdot),q(\cdot)}(\rn)$
is only defined on $(0,\fz)$ while not on $\rn$.
On the other hand, the space $\vhlpq$, in this article,
is defined via the variable Lorentz space
$L^{p(\cdot),q(\cdot)}(\rn)$ (with $q(\cdot)\equiv{\rm a\ constant}\in(0,\fz]$)
from \cite{kv14}, which is not covered by the space $H^{p(\cdot),q(\cdot)}(\rn, A)$
in \cite{abr16}.
Moreover, as was pointed out in \cite[p.\,6]{abr16},
the key tool of \cite{abr16} is the fact that
the set $L_\loc^1(\rn)\cap H^{p(\cdot),q(\cdot)}(\rn, A)$
is dense in $H^{p(\cdot),q(\cdot)}(\rn, A)$.
Therefore, the method used in \cite{abr16} does not work for
$\vhlpq$ in the present article,
due to the lack of a dense function subspace of $H_A^{p(\cdot),\fz}(\rn)$ even when
$p(\cdot)\equiv {\rm a\ constant}\in(0,\fz)$ and
$A:=d\,{\rm I}_{n\times n}$ for some $d\in\rr$ with $|d|\in(1,\fz)$.
\end{enumerate}
\end{remark}

\section{Maximal function characterizations of $\vhlpq$\label{s3}}
\hskip\parindent
In this section, we  characterize $\vhlpq$
in terms of the radial maximal function $M_\varphi^0$ (see \eqref{se10})
or the non-tangential maximal function $M_\varphi$ (see \eqref{se11}).
We begin with the following Definitions \ref{td1} and \ref{td2}
from \cite{mb03}.

\begin{definition}\label{td1}
For any function
$F:\ \rn\times\mathbb{Z}\to [0,\infty)$,
$K\in\mathbb{Z}\cup\{\infty\}$ and $\ell\in\mathbb{Z}$,
the \emph{maximal function $F^{* K}_\ell$} of $F$ with \emph{aperture}
$\ell$ is defined by setting, for any $x\in\rn$,
\begin{align*}
F^{* K}_\ell(x):=
\sup_{k\in\mathbb{Z},\,k\le K}
\sup_{y\in x+B_{k+\ell}}F(y,k).
\end{align*}
\end{definition}

\begin{definition}\label{td2}
Let $K\in\zz$, $L\in[0,\fz)$ and $N\in\nn$. For any $\phi\in\cs$,
the \emph{radial maximal function} $M_\phi^{0(K,L)}(f)$,
the \emph{non-tangential maximal function} $M_\phi^{(K,L)}(f)$ and
the \emph{tangential maximal function} $T_\phi^{N(K,L)}(f)$ of
$f\in\cs'(\rn)$ are, respectively, defined by
setting, for any $x\in\rn$,
\begin{align*}
M_\phi^{0(K,L)}(f)(x):=\sup_{k\in\mathbb{Z},\,k\le K}
|(f\ast\phi_k)(x)|\lf[\max\lf\{1,\rho\lf(A^{-K}x\r)\r\}\r]^
{-L}\lf(1+b^{-k-K}\r)^{-L},
\end{align*}
\begin{align*}
M_\phi^{(K,L)}(f)(x):=\sup_{k\in\mathbb{Z},\,k\le K}
\sup_{y\in x+B_k}|(f\ast\phi_k)(y)|\lf[\max\lf\{1,\rho
\lf(A^{-K}y\r)\r\}\r]^{-L}\lf(1+b^{-k-K}\r)^{-L}
\end{align*}
and
\begin{align*}
T_\phi^{N(K,L)}(f)(x):=\sup_{k\in\mathbb{Z},\,k\le K}
\sup_{y\in\rn}\frac{|(f\ast\phi_k)(y)|}{[\max\lf\{1,\rho
\lf(A^{-k}(x-y)\r)\r\}]^{N}}\frac{(1+b^{-k-K})^{-L}}
{[\max\lf\{1,\rho\lf(A^{-K}y\r)\r\}]^{L}}.
\end{align*}
Furthermore, the \emph{radial grand maximal function} $M_N^{0(K,L)}(f)$
and the \emph{non-tangential grand maximal function}
$M_N^{(K,L)}(f)$ of
$f\in\cs'(\rn)$ are, respectively, defined by
setting, for any $x\in\rn$,
\begin{align*}
M_N^{0(K,L)}(f)(x):=
\sup_{\phi\in\cs_N(\rn)}M_\phi^{0(K,L)}(f)(x)
\end{align*} and
\begin{align*}
M_N^{(K,L)}(f)(x):=
\sup_{\phi\in\cs_N(\rn)}M_\phi^{(K,L)}(f)(x).
\end{align*}
\end{definition}

For any $r\in(0,\fz)$, denote by $L_{\rm loc}^r(\rn)$ the set of all
locally $r$-integrable functions
on $\rn$ and, for any measurable set $E\subset\rn$, by $L^r(E)$
the set of all measurable functions $f$ such that
$$\|f\|_{L^r(E)}:=\lf[\int_E|f(x)|^r\,dx\r]^{1/r}<\fz.$$
Recall that the \emph{Hardy-Littlewood maximal operator}
$M_{{\rm HL}}(f)$ is defined by setting,
for any $f\in L^1_{{\rm loc}}(\rn)$ and $x\in\rn$,
\begin{align}\label{se5}
M_{{\rm HL}}(f)(x):=\sup_{k\in\mathbb{Z}}
\sup_{y\in x+B_k}\frac1{|B_k|}
\int_{y+B_k}|f(z)|\,dz=\sup_{x\in B\in\mathfrak{B}}
\frac1{|B|}\int_B|f(z)|\,dz,
\end{align}
where $\mathfrak{B}$ is as in \eqref{se14}.

Observe that $(\rn,\,\rho,\,dx)$ is a space
of homogeneous type in the sense of
Coifman and Weiss \cite{cw71,cw77}. From this and
\cite[Theorems 5.2 and 4.3]{hhp04}, we deduce the following lemma,
which is an anisotropic version of \cite[Theorem 3.16]{cf13},
the details being omitted.

\begin{lemma}\label{sl1}
Let $p(\cdot)\in C^{\log}(\rn)$.
\begin{enumerate}
\item[{\rm(i)}] If $1\le p_-\le p_+<\fz$, then,
for any $s\in[1,\fz)$, $sp(\cdot)\in C^{\log}(\rn)$
and, for any $f\in L^{sp(\cdot)}(\rn)$,
$$\sup_{\lz\in(0,\fz)}\lf\|\lz\chi_{\{x\in\rn:\ M_{{\rm HL}}(f)(x)>\lz\}}\r\|
_{L^{sp(\cdot)}(\rn)}\le C\|f\|_{L^{sp(\cdot)}(\rn)},$$
where $C$ is a positive constant independent of $f$;

\item[{\rm(ii)}] If $1<p_-\le p_+<\fz$, then,
for any $s\in[1,\fz)$, $sp(\cdot)\in C^{\log}(\rn)$
and, for any $f\in L^{sp(\cdot)}(\rn)$,
$$\|M_{{\rm HL}}(f)\|_{L^{sp(\cdot)}(\rn)}
\le \widetilde{C}\|f\|_{L^{sp(\cdot)}(\rn)},$$
where $\widetilde{C}$ is a positive constant independent of $f$.
\end{enumerate}
\end{lemma}

Moreover, as a simple consequence of \cite[Theorem 4.1]{kv14},
\cite[Theorem 3.1]{yyyz16} and Lemma \ref{sl1}(ii),
we immediately obtain the following boundedness of $\HL$ on $\vlpq$,
which is of independent interest, the details being omitted.

\begin{lemma}\label{sl7}
Let $p(\cdot)\in C^{\log}(\rn)$ satisfy $1<p_-\le p_+<\fz$,
where $p_-$ and $p_+$ are as in \eqref{se3}, and $q\in(0,\fz]$.
Then the Hardy-Littlewood maximal operator $\HL$ is bounded on $\vlpq$.
\end{lemma}

\begin{lemma}\label{tl1}
Let $p(\cdot)\in C^{\log}(\rn),\,N\in(1/p_-,\fz)\cap\mathbb{N}$
and $\phi\in\cs(\rn)$.
Then there exists a positive constant $C$ such that, for any
$K\in\mathbb{Z}$, $L\in[0,\fz)$ and $f\in\cs'(\rn)$,
\begin{align}\label{te1}
\lf\|T_\phi^{N(K,L)}(f)\r\|_{\vlpq}
\le C\lf\|M_\phi^{(K,L)}(f)\r\|_{\vlpq}.
\end{align}
\end{lemma}

\begin{proof}
We first prove that, for any $p(\cdot)\in C^{\log}(\rn)$,
$K\in\mathbb{Z}$ and $\ell\in[\ell',\fz)\cap\mathbb{Z}$,
\begin{align}\label{te5}
\lf\|F_\ell^{* K}\r\|_{\vlpq}
\ls b^{(\ell-\ell')/p_-}\lf\|F_{\ell'}^{* K}\r\|
_{\vlpq},
\end{align}
where $F_\ell^{* K}$ is as in Definition \ref{td1} and,
for any $\phi\in\cs(\rn)$, $f\in\cs'(\rn)$,
$k\in\zz$ and $y\in\rn$,
\begin{align*}
F(y,k):=|(f\ast\phi_k)(y)|\max\lf[\lf\{1,\rho
\lf(A^{-K}y\r)\r\}\r]^{-L}\lf(1+b^{-k-K}\r)^{-L}.
\end{align*}
Indeed,
by a proof similar to that of \cite[p.\,42, Lemma 7.2]{mb03},
we easily find that, for any $\ell,\,\ell'\in\mathbb{Z}$ with $\ell\geq\ell'$
and $\lz\in(0,\fz)$,
\begin{align}\label{te7}
\lf\{x\in\rn:\ F^{* K}_\ell(x)>\lz\r\}\subset
\lf\{x\in\rn:\ \HL(\chi_{\Omega_\lz})(x)\ge b^{\ell'-\ell-\tau}\r\},
\end{align}
where, for any $\lz\in(0,\fz)$,
$\Omega_\lz:=\{x\in\rn:\ F^{* K}_{\ell'}(x)>\lz\}$. Then, by \eqref{te7},
Lemma \ref{sl1}(i) and Remark \ref{sr1}(i), we know that
\begin{align}\label{te8}
\lf\|\chi_{\{x\in\rn:\ F^{* K}_\ell(x)>\lz\}}\r\|_{\lv}
&\le\lf\|\chi_{\{x\in\rn:\
\HL(\chi_{\Omega_\lz})(x)\ge b^{\ell'-\ell-\tau}\}}\r\|_{\lv}\\
&=\lf\|\chi_{\{x\in\rn:\ \HL(\chi_{\Omega_\lz})(x)\ge
b^{\ell'-\ell-\tau}\}}\r\|^{1/p_-}_{L^{\frac{p(\cdot)}{p_-}}(\rn)}\noz\\
&\sim b^{\frac{\ell-\ell'}{p_-}}\lf\|b^{\ell'-\ell-\tau}
\chi_{\{x\in\rn:\ \HL(\chi_{\Omega_\lz})(x)\ge
b^{\ell'-\ell-\tau}\}}\r\|^{1/p_-}_{L^{\frac{p(\cdot)}{p_-}}(\rn)}\noz\\
&\ls b^{\frac{\ell-\ell'}{p_-}}\lf\|\chi_{\Omega_\lz}
\r\|^{1/p_-}_{L^{\frac{p(\cdot)}{p_-}}(\rn)}
\sim b^{\frac{\ell-\ell'}{p_-}}\lf\|\chi_{\Omega_\lz}\r\|_{\lv},\noz
\end{align}
which, together with the definition of $\Omega_\lz$ and Definition \ref{sd4},
further implies \eqref{te5}.

On the other hand, by \cite[(4.7)]{lyy16},
for any $N\in\mathbb{N}$, $K\in\mathbb{Z}$, $L\in[0,\fz)$,
$f\in\cs'(\rn)$ and $x\in\rn$, we have
\begin{align}\label{te6}
T_\phi^{N(K,L)}(f)(x)
\le\sum_{j=0}^\infty F^{* K}_{j+1}(x)b^{-jN}.
\end{align}

Now we show \eqref{te1}. By \eqref{te6}, the Aoki-Rolewicz theorem
(see \cite{ta42, sr57}), \eqref{te8} and the fact that
$N\in(1/p_-,\fz)\cap\mathbb{N}$, it is easy to see that
there exists $\upsilon\in(0,1]$ such that
\begin{align*}
\lf\|T_\phi^{N(K,L)}(f)\r\|_{\vlpq}^\upsilon
&\le \sum_{j=0}^\infty b^{-jN\upsilon}\lf\|F^{* K}
_{j+1}\r\|_{\vlpq}^\upsilon\\
&\ls \sum_{j=0}^\infty b^{-jN\upsilon}b^{(j+1)
\upsilon/p_-}\lf\|F^{* K}_0\r\|_{\vlpq}^\upsilon
\ls\lf\|M_\phi^{(K,L)}(f)\r\|_{\vlpq}^\upsilon,\noz
\end{align*}
which implies \eqref{te1} and
hence completes the proof of Lemma \ref{tl1}.
\end{proof}

The following Lemmas \ref{tl2} and \ref{tl3} are just
\cite[p.\,45, Lemma 7.5 and p.\,46, Lemma 7.6]{mb03}, respectively.

\begin{lemma}\label{tl2}
Let $\phi\in\cs(\rn)$ and
$\int_{\rn}\phi(x)\,dx\neq0$. Then,
for any given $N\in\mathbb{N}$ and $L\in[0,\fz)$, there
exist an $I\in\mathbb{N}$ and a positive constant $C_{(N,L)}$,
depending on $N$ and $L$,
such that, for any $K\in\zz_+$, $f\in\cs'(\rn)$ and $x\in\rn$,
$$M_I^{0(K,L)}(f)(x)\le C_{(N,L)}T_\phi^{N(K,L)}(f)(x).$$
\end{lemma}

\begin{lemma}\label{tl3}
Let $\phi\in\cs(\rn)$ and $\int_{\rn}\phi(x)\,dx\neq0$.
Then, for any given $M\in(0,\fz)$ and $K\in\zz_+$,
there exist $L\in(0,\fz)$ and a positive constant $C_{(K,\,M)}$,
depending on $K$ and $M$, such that,
for any $f\in\cs'(\rn)$ and $x\in\rn$,
\begin{align}\label{te10}
M_\phi^{(K,L)}(f)(x)
\le C_{(K,\,M)}\lf[\max\lf\{1,\rho(x)\r\}\r]^{-M}.
\end{align}
\end{lemma}

Now we state the main result of this section as follows.

\begin{theorem}\label{tt1}
Suppose that $p(\cdot)\in C^{\log}(\rn)$ and $\phi\in\cs(\rn)$
with $\int_{\rn}\phi(x)\,dx\neq0$. Then, for any $f\in\cs'(\rn)$,
the following statements are mutually equivalent:
\begin{align}\label{te2}
f\in\vhlpq;   \end{align}
\begin{align}\label{te3}
M_\phi(f)\in\vlpq;   \end{align}
\begin{align}\label{te4}
M_\phi^0(f)\in\vlpq.   \end{align}
In this case, it holds true that
\begin{align*}
\|f\|_{\vhlpq}\le C_1\lf\|M_\phi^0(f)
\r\|_{\vlpq}\le C_1\lf\|M_\phi(f)\r\|_
{\vlpq}\le C_2\|f\|_{\vhlpq},
\end{align*}
where $C_1$ and $C_2$ are two positive constants independent
of $f$.
\end{theorem}

\begin{proof}
Obviously, \eqref{te2} implies \eqref{te3} and
\eqref{te3} implies \eqref{te4}.
Thus, to prove Theorem \ref{tt1}, it suffices
to show that \eqref{te3} implies \eqref{te2}
and that \eqref{te4} implies \eqref{te3}.

We first prove that \eqref{te3} implies \eqref{te2}.
To this end, notice that, by Lemma
\ref{tl2} with $N\in(1/p_-,\fz)\cap\mathbb{N}$ and
$L=0$, we find that there exists an $I\in\nn$ such that
$M_I^{0(K,0)}(f)(x)\ls T_\varphi^{N(K,0)}(f)(x)$
for any $x\in\rn$, $f\in\cs'(\rn)$ and $K\in\zz_+$.
From this and Lemma \ref{tl1}, we further deduce that,
for any $K\in\mathbb{Z}_+$ and $f\in\cs'(\rn)$,
\begin{align}\label{te9}
\lf\|M_I^{0(K,0)}(f)\r\|_{\vlpq}\ls
\lf\|M_\phi^{(K,0)}(f)\r\|_{\vlpq}.
\end{align}
Letting $K\to\infty$ in \eqref{te9}, by
Proposition \ref{sl6}, we know that
$$\lf\|M_I^0(f)\r\|_{\vlpq}
\ls\lf\|M_\phi(f)\r\|_{\vlpq},$$
which shows that \eqref{te3} implies \eqref{te2}.

Now we show that \eqref{te4} implies \eqref{te3}.
Assume now that $M_\phi^0(f)\in \vlpq$.
By Lemma \ref{tl3} via choosing $M\in(1/p_-,\fz)$, we find that there exists $L\in(0,\fz)$
such that \eqref{te10} holds true, which further implies that
$M_\phi^{(K,L)}(f)\in \vlpq$ for any $K\in\zz_+$.
Indeed, when $q\in(0,\fz)$ and $M\in(1/p_-,\fz)$, we have
\begin{align}\label{te20}
\lf\|M_\phi^{(K,L)}(f)\r\|_{\vlpq}^q&=\int_0^\fz\lz^{q-1}
\lf\|\chi_{\{x\in\rn:\ M_\phi^{(K,L)}(f)(x)>\lz\}}\r\|_{\lv}^q\,d\lz\\
&\ls\int_0^1\lz^{p-1}\lf\|\chi_{\{x\in\rn:\
M_\phi^{(K,L)}(f)(x)>\lz\}}\r\|_{L^{p(\cdot)}(B_1)}^q\,d\lz\noz\\
&\hs\hs+\sum_{k=1}^\fz\int_0^{b^{-kM}}\lz^{q-1}
\lf\|\chi_{\{x\in\rn:\ M_\phi^{(K,L)}(f)(x)>\lz\}}\r\|_
{L^{p(\cdot)}(B_{k+1}\setminus B_k)}^q\,d\lz\noz\\
&\ls\sum_{k=0}^\fz b^{-kMq}b^{\frac{(k+1)q}{p_-}}
\sim1.\noz
\end{align}
Clearly, when $q=\fz$, \eqref{te20} still holds true.
Thus, $M_\phi^{(K,L)}(f)\in \vlpq$.

On the other hand, by Lemmas \ref{tl2} and \ref{tl1},
we know that, for any $L\in(0,\fz)$, there exists  some $I\in\nn$
such that, for any $K\in\zz_+$ and $f\in\cs'(\rn)$,
$$\lf\|M_I^{0(K,L)}(f)\r\|_{\vlpq}\le C_3
\lf\|M_\phi^{(K,L)}(f)\r\|_{\vlpq},$$
where $C_3$ is a positive constant independent of $K$ and $f$.
For any fixed $K\in\zz_+$, let
\begin{align*}
\Omega_K:=\lf\{x\in\rn:\ M_I^{0(K,L)}(f)(x)\le C_4
M_\phi^{(K,L)}(f)(x)\r\}
\end{align*}
with $C_4:=2C_3$. Then
\begin{align}\label{te12}
\lf\|M_\phi^{(K,L)}(f)\r\|_{\vlpq}\ls
\lf\|M_\phi^{(K,L)}(f)\r\|_{L^{p(\cdot),q}(\Omega_K)},
\end{align}
because
$$\lf\|M_\phi^{(K,L)}(f)\r\|_{L^{p(\cdot),q}(\Omega_K
^\complement)}\le C_4^{-1}\lf\|M_I^{0(K,L)}(f)\r\|
_{L^{p(\cdot),q}(\Omega_K^\complement)}\le C_3/C_4
\lf\|M_\phi^{(K,L)}(f)\r\|_{\vlpq}.$$

For any given $L\in[0,\fz)$,
by an argument similar to that used in the proof of \cite[(4.17)]{lyy16},
we conclude that, for any $t\in(0,p_-)$, $K\in\zz_+$, $f\in\cs'(\rn)$
 and $x\in\Omega_K$,
\begin{align}\label{te13}
M_\phi^{(K,L)}(f)(x)\ls\lf\{M_{{\rm HL}}
\lf(\lf[M_\phi^{0(K,L)}(f)\r]^t\r)(x)\r\}^{1/t}.
\end{align}

Then, by \eqref{te12}, Lemma \ref{sl4}, \eqref{te13} and
Lemma \ref{sl7}, for any $K\in\zz_+$ and $f\in\cs'(\rn)$, we
further find that
\begin{align*}
\lf\|M_\phi^{(K,L)}(f)\r\|_{\vlpq}^t
&\ls\lf\|M_\phi^{(K,L)}(f)\r\|_{L^{p(\cdot),q}(\Omega_K)}^t
\sim\lf\|\lf[M_\phi^{(K,L)}(f)\r]^t\r\|_{L^{
\frac {p(\cdot)}t,\frac qt}(\Omega_K)}\\
&\ls\lf\|M_{{\rm HL}}\lf(\lf[M_\phi^{0(K,L)}(f)\r]^
t\r)\r\|_{L^{\frac {p(\cdot)}t,\frac qt}(\Omega_K)}\\
&\ls\lf\|\lf[M_\phi^{0(K,L)}(f)\r]^t\r\|_{L^{
\frac {p(\cdot)}t,\frac qt}(\Omega_K)}
\sim\lf\|M_\phi^{0(K,L)}(f)\r\|_{\vlpq}^t,
\end{align*}
which, together with the fact that
$M_\phi^{(K,L)}(f)$ and $M_\phi^{0(K,L)}(f)$
converge pointwisely and monotonically, respectively, to $M_\phi (f)$
and $M_\phi^0(f)$ as $K\to\infty$ and Proposition \ref{sl6},
implies that
\begin{align*}
\lf\|M_\phi (f)\r\|_{\vlpq}
\ls\lf\|M_\phi^0(f)\r\|_{\vlpq}.
\end{align*}
This shows that \eqref{te4} implies \eqref{te3} and hence finishes
the proof of Theorem \ref{tt1}.
\end{proof}

\section{Atomic characterization of $\vhlpq$\label{s4}}
\hskip\parindent
In this section, we establish the atomic characterization of $\vhlpq$.
We begin with the following notion of anisotropic $(p(\cdot),r,s)$-atoms.

\begin{definition}\label{fd1}
Let $p(\cdot)\in\cp(\rn)$,
$r\in(1,\fz]$ and
\begin{align}\label{fe1}
s\in\lf\lfloor\lf(\dfrac1{p_-}-1\r)
\dfrac{\ln b}{\ln\lambda_-}\r\rfloor\cap\zz_+.
\end{align}
A measurable function $a$ on $\rn$
is called an \emph{anisotropic $(p(\cdot),q,s)$-atom} if
\begin{enumerate}
\item[{\rm (i)}] $\supp a \st B$, where
$B\in\mathfrak{B}$ and $\mathfrak{B}$ is as in \eqref{se14};

\item[{\rm (ii)}] $\|a\|_{L^r(\rn)}\le \frac{|B|^{1/r}}{\|\chi_B\|_{\lv}}$;

\item[{\rm (iii)}] $\int_{\mathbb R^n}a(x)x^\az\,dx=0$ for any $\az\in{\zz}_+^n$
with $|\az|\le s$.
\end{enumerate}
\end{definition}

For the presentation simplicity, throughout this article, we call an anisotropic
$(p(\cdot),r,s)$-atom simply by a $(p(\cdot),r,s)$-atom.
Now, via $(p(\cdot),r,s)$-atoms, we introduce the notion
of the anisotropic variable atomic Hardy-Lorentz space
$\vahlpq$ as follows.

\begin{definition}\label{fd2}
Let $p(\cdot)\in C^{\log}(\rn)$, $q\in(0,\fz]$, $r\in(1,\fz]$, $s$ be
as in \eqref{fe1} and $A$ be a dilation.
The \emph{anisotropic variable atomic Hardy-Lorentz space}
$\vahlpq$ is defined to be the set of all distributions
$f\in\cs'(\rn)$ satisfying that there exist a sequence of $(p(\cdot),r,s)$-atoms,
$\{a_i^k\}_{i\in\mathbb{N},\,k\in\mathbb{Z}}$,
supported, respectively, on
$\{B_i^k\}_{i\in\mathbb{N},\,k\in\mathbb{Z}}\subset\mathfrak{B}$
and a positive constant $\widetilde{C}$ such that
$\sum_{i\in\mathbb{N}}\chi_{A^{j_0}B_i^k}(x)\le \widetilde{C}$
for any $x\in\rn$ and $k\in\mathbb{Z}$,
with some $j_0\in\zz\setminus\nn$, and
\begin{align}\label{fe2}
f=\sum_{k\in\mathbb{Z}}\sum_{i\in\mathbb{N}}\lambda_i^ka_i^k
\quad\mathrm{in}\quad\cs'(\rn),
\end{align}
where $\lambda_i^k\sim2^k\|\chi_{B_i^k}\|_{\lv}$ for any
$k\in\mathbb{Z}$ and $i\in\mathbb{N}$ with the equivalent
positive constants independent of $k$ and $i$.

Moreover, for any $f\in\vahlpq$, define
\begin{align*}
\|f\|_{\vahlpq}:=
{\rm inf}\lf[\sum_{k\in\zz}\lf\|\lf\{\sum_{i\in\nn}
\lf[\frac{\lz_i^k\chi_{B_i^k}}{\|\chi_{B_i^k}\|_{\lv}}\r]^
{\underline{p}}\r\}^{1/\underline{p}}\r\|_{\lv}^q\r]^{\frac 1q}
\end{align*}
with the usual interpretation for $q=\fz$, where the
infimum is taken over all decompositions of $f$ as above.
\end{definition}

In order to establish the atomic decomposition of
$\vhlpq$, we need several technical lemmas as follows.
First, observe that $(\rn,\,\rho,\,dx)$ is an RD-space
(see \cite{hmy08,zy11}). From this and
\cite[Theorem 2.7]{zsy16}, we deduce the  following
Fefferman-Stein vector-valued inequality of the maximal
operator $\HL$ on the variable Lebesgue space $\lv$,
the details being omitted.

\begin{lemma}\label{fl1}
Let $r\in(1,\fz]$.
Assume that $p(\cdot)\in C^{\log}(\rn)$ satisfies $1<p_-\le p_+<\fz$.
Then there exists a positive
constant $C$ such that, for any
sequence $\{f_k\}_{k\in\nn}$ of measurable functions,
$$\lf\|\lf\{\sum_{k\in\nn}
\lf[\HL(f_k)\r]^r\r\}^{1/r}\r\|_{\lv}
\le C\lf\|\lf(\sum_{k\in\nn}|f_k|^r\r)^{1/r}\r\|_{\lv}$$
with the usual modification made when $r=\fz$,
where $\HL$ denotes the Hardy-Littlewood maximal operator as in \eqref{se5}.
\end{lemma}

\begin{remark}\label{fr1}
\begin{enumerate}
\item[{\rm (i)}]
Let $p(\cdot)\in C^{\log}(\rn)$  and $i\in\zz_+$. Then, by Lemma \ref{fl1}
and the fact that, for any dilated ball $B\in\mathfrak{B}$ and $r\in(0,\underline{p})$,
$\chi_{A^iB}\le b^{\frac{i}{r}}[\HL(\chi_B)]^{\frac{1}{r}}$,
we conclude that there exists a positive
constant $C$ such that, for any sequence $\{B^{(k)}\}_{k\in\nn}\st\mathfrak{B}$,
$$\lf\|\sum_{k\in\nn}\chi_{A^i B^{(k)}}\r\|_{\lv}\le
C b^{\frac{i}{r}}\lf\|\sum_{k\in\nn}\chi_{B^{(k)}}\r\|_{\lv}.$$

\item[{\rm (ii)}]
Let $f\in\vahlpq$. By the definition of $\vahlpq$,
we know that there exists a sequence $\{a_i^k\}_{i\in\nn,k\in\zz}$ of
$(p(\cdot),r,s)$-atoms supported, respectively, on
$\{B_i^k\}_{i\in\mathbb{N},\,k\in\mathbb{Z}}\subset\mathfrak{B}$,
satisfying that, for any $k\in\zz$ and $x\in\rn$,
$\sum_{i\in\mathbb{N}}\chi_{A^{j_0}B_i^k}(x)\le \widetilde{C}$ with
some $j_0\in\zz\setminus\nn$ and $\widetilde{C}$ being a positive constant
independent of $k$ and $x$ such that $f$ admits a decomposition as in \eqref{fe2}
with $\lambda_i^k\sim2^k\|\chi_{B_i^k}\|_{\lv}$ for any
$k\in\mathbb{Z}$ and $i\in\mathbb{N}$, where the equivalent
positive constants are independent of $k$ and $i$, and
\begin{align*}
\|f\|_{\vahlpq}\sim\lf[\sum_{k\in\zz}\lf\|\lf\{\sum_{i\in\nn}
\lf[\frac{\lz_i^k\chi_{B_i^k}}{\|\chi_{B_i^k}\|_{\lv}}\r]^
{\underline{p}}\r\}^{1/\underline{p}}\r\|_{\lv}^q\r]^{\frac 1q}
\end{align*}
with the equivalent positive constants independent of $f$.
Moreover, by
$\sum_{i\in\mathbb{N}}\chi_{A^{j_0}B_i^k}\le \widetilde{C}$
for any $k\in\zz$, the definition of $\lz_i^k$
and (i) of this remark, we further conclude that
\begin{align*}
\|f\|_{\vahlpq}&\sim\lf[\sum_{k\in\zz}2^{kq}\lf\|\lf(\sum_{i\in\nn}
\chi_{B_i^k}\r)^{1/{\underline{p}}}\r\|_{\lv}^q\r]^{\frac1q}\\
&\sim\lf[\sum_{k\in\zz}2^{kq}\lf\|\lf(\sum_{i\in\nn}
\chi_{A^{j_0}B_i^k}\r)^{1/{\underline{p}}}\r\|_{\lv}^q\r]^{\frac1q}\\
&\sim\lf[\sum_{k\in\zz}2^{kq}\lf\|\sum_{i\in\nn}
\chi_{A^{j_0}B_i^k}\r\|_{\lv}^q\r]^{\frac1q}
\sim\lf[\sum_{k\in\zz}2^{kq}\lf\|\sum_{i\in\nn}
\chi_{B_i^k}\r\|_{\lv}^q\r]^{\frac1q},
\end{align*}
where the equivalent positive constants are independent of $f$.
\end{enumerate}
\end{remark}

We also need the following useful technical lemma,
whose proof is similar to that of \cite[Lemma 4.1]{s10},
the details being omitted.

\begin{lemma}\label{fl2}
Let $p(\cdot)\in C^{\log}(\rn)$, $t\in(0,\underline{p}]$ and
$r\in[1,\fz]\cap(p_+,\fz]$. Then there exists
a positive constant $C$ such that, for any sequence
$\{B^{(k)}\}_{k\in\nn}\st\mathfrak{B}$ of dilated balls,
numbers $\{\lz_k\}_{k\in\nn}\subset\mathbb C$
and measurable functions $\{a_k\}_{k\in\nn}$ satisfying that,
for each $k\in\nn$, $\supp a_k\st B^{(k)}$ and
$\|a_k\|_{L^r(\rn)}\le |B^{(k)}|^{1/r}$,
it holds true that
$$\lf\|\lf(\sum_{k\in\nn}\lf|\lz_k a_k\r|^t\r)^{\frac1t}
\r\|_{\lv}\le C\lf\|\lf(\sum_{k\in\nn}\lf|\lz_k\chi_{B^{(k)}}\r|^t\r)
^{\frac1t}\r\|_{\lv}.$$
\end{lemma}

The following Proposition \ref{sp1} and Lemma \ref{fl3}
are just \cite[p.\,17, Proposition 3.10 and p.\,9, Lemma 2.7 ]{mb03},
respectively.

\begin{proposition}\label{sp1}
For any given $N\in\mathbb{N}$,
there exists a positive constant $C_{(N)}$,
depending only on $N$, such that,
for any $f\in\cs'(\rn)$ and $x\in\rn$,
\begin{equation*}
M_N^0(f)(x)\le M_N(f)(x)
\le C_{(N)}M_N^0(f)(x),
\end{equation*}
where $M_N^0(f)$ and $M_N(f)$ are as in Definition \ref{sd1}.
\end{proposition}

\begin{lemma}\label{fl3}
Let $\Omega\subset\rn$ be an open set with $|\Omega|<\infty$.
Then, for any $m\in\zz_+$, there exist a sequence of points,
$\{x_k\}_{k\in\mathbb{N}}\subset\Omega$, and a sequence of
integers, $\{\ell_k\}_{k\in\mathbb{N}}\subset\mathbb{Z}$,
such that

\begin{enumerate}
\item[{\rm(i)}] $\Omega=\bigcup_{k\in\mathbb{N}}
(x_k+B_{\ell_k})$;

\item[{\rm(ii)}] $\lf\{x_k+B_{\ell_k-\tau}\r\}_{k\in\nn}$ are pairwise
disjoint, where $\tau$ is as in \eqref{se1} and \eqref{se2};

\item[{\rm(iii)}] for each $k\in\mathbb{N},\
(x_k+B_{\ell_k+m})\cap\Omega^\complement
=\emptyset$, but $(x_k+B_{\ell_k+m+1})\cap
\Omega^\complement\neq\emptyset$;

\item[{\rm(iv)}] $if\ (x_i+B_{\ell_i+m-2\tau})\cap
(x_j+B_{\ell_j+m-2\tau})\neq\emptyset$, then
$|\ell_i-\ell_j|\le\tau$;

\item[{\rm(v)}] for any $i\in\mathbb{N}$,
$\sharp\{j\in\mathbb{N}:\ (x_i+B_
 {\ell_i+m-2\tau})\cap(x_j+B_{\ell_j+m-2\tau})
 \neq\emptyset\}\le R$,
where $R$ is a positive constant
independent of $\Omega$ and $i$.
\end{enumerate}
\end{lemma}

Now, it is a position to state the main result of this section.

\begin{theorem}\label{ft1}
Let $p(\cdot)\in C^{\log}(\rn)$, $q\in(0,\fz]$,
$r\in(\max\{p_+,1\},\fz]$ with $p_+$ as in \eqref{se3} and $s$ be as in
\eqref{fe1}.
Then $\vhlpq=\vahlpq$ with equivalent quasi-norms.
\end{theorem}

\begin{proof}
First, we show that
\begin{align}\label{fe4}
\vahlpq\subset\vhlpq.
\end{align}
To this end, for any $f\in\vahlpq$, by Remark \ref{fr1}(ii),
we find that there exists a sequence of $(p(\cdot),r,s)$-atoms,
$\{a_i^k\}_{i\in\mathbb{N},\,k\in\mathbb{Z}}$,
supported, respectively, on
$\{B_i^k\}_{i\in\mathbb{N},\,k\in\mathbb{Z}}\st\mathfrak{B}$ such that
$$f=\sum_{k\in\mathbb{Z}}
\sum_{i\in\mathbb{N}}\lambda_i^ka_i^k\quad \mathrm{in}\quad \cs'(\rn),$$
where $\lambda_i^k\sim2^k\|\chi_{B_i^k}\|_{\lv}$ for any
$k\in\mathbb{Z}$ and
$i\in\mathbb{N}$, $\sum_{i\in\mathbb{N}}
\chi_{A^{j_0}B_i^k}(x)\ls1$ with some $j_0\in\zz\setminus\nn$
for any $x\in\rn$ and $k\in\mathbb{Z}$, and
\begin{align}\label{fe5}
\|f\|_{\vahlpq}
\sim\lf[\sum_{k\in\zz}2^{kq}\lf\|\sum_{i\in\nn}
\chi_{B_i^k}\r\|_{\lv}^q\r]^{\frac1q}.
\end{align}
For the notational simplicity, in what follows of this proof,
for any $f\in\cs'(\rn)$, we \emph{denote $M_N^0(f)$ simply by $M(f)$}
when $N\in\mathbb{N}\cap[\lfloor(\frac1{\underline{p}}-1)
\frac{\ln b}{\ln\lambda_-}\rfloor+2,\fz)$.
To prove $f\in\vhlpq$, by Definition \ref{sd3} and Lemma \ref{sl3*},
it suffices to show that
\begin{align*}
\lf[\sum_{k\in\zz}2^{kq}
\lf\|\chi_{\{x\in\rn:\ M(f)(x)>2^k\}}\r\|_{\lv}^q\r]^{\frac1q}
\ls\|f\|_{\vahlpq}.
\end{align*}
For any fixed $k_0\in\zz$, we write
$$f=\sum_{k=-\fz}^{k_0-1}\sum_{i\in\mathbb{N}}\lambda_i^ka_i^k
+\sum_{k=k_0}^{\fz}\sum_{i\in\mathbb{N}}\cdots
=:f_1+f_2.$$
Then, by Remark \ref{sr1}(i), we know that
\begin{align}\label{fe6}
&\lf\|\chi_{\{x\in\rn:\ M(f)(x)>2^{k_0}\}}\r\|_{\lv}\\
&\hs\ls\lf\|\chi_{\{x\in\rn:\ M(f_1)(x)>2^{k_0-1}\}}\r\|_{\lv}
+\lf\|\chi_{\{x\in E_{k_0}:\ M(f_2)(x)>2^{k_0-1}\}}\r\|_{\lv}\noz\\
&\hs\hs+\lf\|\chi_{\{x\in (E_{k_0})^\com:\ M(f_2)(x)>2^{k_0-1}\}}\r\|_{\lv}\noz\\
&\hs=:{\rm J}_1+{\rm J}_2+{\rm J}_3,\noz
\end{align}
where $E_{k_0}:=\bigcup_{k=k_0}^\fz\bigcup_{i\in\nn}A^\tau B_i^k$.

For ${\rm J}_1$, clearly, we have
\begin{align}\label{fe7}
{\rm J}_1&\ls\lf\|\chi_{\{x\in\rn:\ \sum_{k=-\fz}^{k_0-1}\sum_{i\in\mathbb{N}}
\lambda_i^kM(a_i^k)(x)\chi_{A^\tau B_i^k}(x)>2^{k_0-2}\}}\r\|_{\lv}\\
&\hs+\lf\|\chi_{\{x\in\rn:\ \sum_{k=-\fz}^{k_0-1}\sum_{i\in\mathbb{N}}
\lambda_i^kM(a_i^k)(x)\chi_{(A^\tau B_i^k)^\com}(x)>2^{k_0-2}\}}\r\|_{\lv}\noz\\
&=:{\rm J}_{1,1}+{\rm J}_{1,2}.\noz
\end{align}

To deal with ${\rm J}_{1,1}$,
by the H\"older inequality, we find that, for any given
$t\in(0,\min\{\underline{p},q\})$,
$\q1\in(1,\min\{\frac{r}{\max\{p_+,1\}},\frac{1}{t}\})$,
$\dz\in(0,1-\frac1{\q1})$ and any $x\in\rn$,
\begin{align*}
&\sum_{k=-\fz}^{k_0-1}\sum_{i\in\nn}\lik M(\aik)
(x)\chi_{A^\tau\Bik}(x)\\
&\hs\le \lf(\sum_{k=-\fz}^{k_0-1}2^{k\dz{\q1}'}\r)
^{1/{\q1}'}\lf\{\sum_{k=-\fz}^{k_0-1}2^{-k\dz{\q1}}
\lf[\sum_{i\in\nn}\lik M(\aik)(x)\chi_{A^\tau\Bik}(x)\r]
^{\q1}\r\}^{1/{\q1}}\\
&\hs=\frac{2^{k_0\dz}}{(2^{\dz{\q1}'}-1)^{1/{\q1}'}}
\lf\{\sum_{k=-\fz}^{k_0-1}2^{-k\dz{\q1}}
\lf[\sum_{i\in\nn}\lik M(\aik)(x)\chi_{A^\tau\Bik}(x)\r]
^{\q1}\r\}^{1/{\q1}},
\end{align*}
where ${\q1}'$ denotes the conjugate index of $\q1$, namely,
$\frac{1}{\q1}+\frac{1}{{\q1}'}=1$. By this, the facts that
$\q1 t<1$ and $M(f)(x)\ls\HL(f)(x)$
for any $x\in\rn$, Remark \ref{sr1}(i) and \cite[Theorem 2.61]{cf13},
we further conclude that
\begin{align*}
{\rm J}_{1,1}
&\ls\lf\|\chi_{\{x\in\rn:\ \frac{2^{k_0\dz}}
{(2^{\dz{\q1}'}-1)^{1/{\q1}'}}
\{\sum_{k=-\fz}^{k_0-1}2^{-k\dz{\q1}}
[\sum_{i\in\nn}\lik M(\aik)(x)\chi_{A^\tau\Bik}(x)]
^{\q1}\}^{1/{\q1}}>
2^{k_0-2}\}}\r\|_{\lv}\\
&\ls 2^{-k_0\q1(1-\dz)}\lf
\|\sum_{k=-\fz}^{k_0-1}2^{-k\dz{\q1}}\lf
[\sum_{i\in\nn}\lik M(\aik)\chi_{A^\tau\Bik}\r]
^{\q1}\r\|_{\lv}\\
&\ls 2^{-k_0\q1(1-\dz)}\lf\|\sum_{k=-\fz}^{k_0-1}
2^{(1-\dz)k{\q1}t}\sum_{i\in\nn}\lf
[\lf\|\chi_{\Bik}\r\|_{\lv}\HL(\aik)\chi_{A^\tau\Bik}\r]
^{\q1 t}\r\|
^{\frac{1}{t}}_{L^{\frac{p(\cdot)}
t}(\rn)}\\
&\ls 2^{-k_0\q1(1-\dz)}\lf[\sum_{k=-\fz}
^{k_0-1}2^{(1-\dz)k{\q1}t}\lf\|\lf\{\sum_{i\in\nn}
\lf[\lf\|\chi_{\Bik}\r\|_{\lv}\HL(\aik)\chi_{A^\tau\Bik}\r]^
{\q1t}\r\}^{\frac{1}{t}}\r\|^{t}_{\lv}\r]^{\frac{1}{t}}.
\end{align*}
On the other hand, since $r>1$, then,
from the boundedness of $\HL$ on $L^{r/\q1}(\rn)$, we deduce that,
for any $k\in\zz$ and $i\in\nn$,
\begin{align*}
\lf\|\lf[\lf\|\chi_{\Bik}\r\|_{\lv}
\HL(\aik)\chi_{A^\tau\Bik}\r]^{\q1}\r\|_{L^{r/\q1}(\rn)}
&\ls\lf\|\chi_{\Bik}\r\|_{\lv}^{\q1}
\lf\|\HL(\aik)\chi_{A^\tau\Bik}\r\|^{\q1}_{L^r(\rn)}
\ls\lf|\Bik\r|^{\frac{\q1}{r}},
\end{align*}
which, combined with
Lemma \ref{fl2}, Remark \ref{fr1}(i), the fact that $t<q$
and the H\"older inequality, implies that
\begin{align*}
{\rm J}_{1,1}
&\ls 2^{-k_0\q1(1-\dz)}\lf[\sum_{k=-\fz}
^{k_0-1}2^{(1-\dz)k{\q1}t}\lf\|\lf(\sum_{i\in\nn}
\chi_{A^\tau\Bik}\r)^{\frac{1}{t}}\r\|^{t}_{\lv}\r]^{\frac{1}{t}}\\
&\ls 2^{-k_0\q1(1-\dz)}\lf\{\sum_{k=-\fz}^{k_0-1}
2^{[(1-\dz){\q1}-\varepsilon]kt}2^{\varepsilon kt}\lf\|\lf(\sum_{i\in\nn}
\chi_{A^{j_0}\Bik}\r)^{\frac{1}{t}}\r\|^{t}_{\lv}\r\}^{\frac{1}{t}}\\
&\ls 2^{-k_0\q1(1-\dz)}\lf\{\sum_{k=-\fz}^{k_0-1}
2^{[(1-\dz){\q1}-\varepsilon]kt}2^{\varepsilon kt}\lf\|\sum_{i\in\nn}
\chi_{A^{j_0}\Bik}\r\|^{t}_{\lv}\r\}^{\frac{1}{t}}\\
&\ls 2^{-k_0\q1(1-\dz)}\lf\{\sum_{k=-\fz}^{k_0-1}
2^{[(1-\dz){\q1}-\varepsilon]kt\frac q{q-t}}\r\}^{\frac{q-t}{qt}}
\lf[\sum_{k=-\fz}^{k_0-1}2^{\varepsilon kq}\lf\|\sum_{i\in\nn}
\chi_{\Bik}\r\|^{q}_{\lv}\r]^{\frac{1}{q}}\\
&\sim 2^{-\varepsilon k_0}\lf[\sum_{k=-\fz}^{k_0-1}
2^{\varepsilon kq}\lf\|\sum_{i\in\nn}
\chi_{\Bik}\r\|^{q}_{\lv}\r]^{\frac{1}{q}},
\end{align*}
where $\varepsilon\in(1,(1-\dz)\q1)$. From this and
\eqref{fe5}, we deduce that
\begin{align}\label{fe8}
\lf[\sum_{k_0\in\zz}2^{k_0q}({\rm J}_{1,1})^q\r]^{\frac1q}
&\ls\lf[\sum_{k_0\in\zz}2^{k_0q}2^{-\varepsilon k_0q}
\sum_{k=-\fz}^{k_0-1}2^{\varepsilon kq}
\lf\|\sum_{i\in\nn}\chi_{\Bik}\r\|^{q}_{\lv}\r]^{\frac1q}\\
&\sim\lf[\sum_{k\in\zz}\sum_{k_0=k+1}^\fz
2^{(1-\varepsilon) k_0q}2^{\varepsilon kq}
\lf\|\sum_{i\in\nn}\chi_{\Bik}\r\|^{q}_{\lv}\r]^{\frac1q}\noz\\
&\sim\lf[\sum_{k\in\zz}2^{kq}
\lf\|\sum_{i\in\nn}\chi_{\Bik}\r\|^{q}_{\lv}\r]^{\frac1q}
\sim\|f\|_{\vahlpq}.\noz
\end{align}

To estimate ${\rm J}_{1,2}$, for any $i\in\nn$, $k\in\zz$
and $x\in (A^\tau\Bik)^\complement$, by an argument
similar to that used in the proof of \cite[(3.27)]{lyy16},
we have
\begin{align}\label{fe9}
M(\aik)(x)\ls\lf\|\chi_{\Bik}\r\|_{\lv}^{-1}\frac
{|\Bik|^\beta}{[\rho(x-\xik)]^\beta}
\ls\lf\|\chi_{\Bik}\r\|_{\lv}^{-1}\lf[\HL(\chi_{\Bik})(x)\r]^\bz,
\end{align}
where, for any $i\in\nn$, $k\in\zz$, $\xik$ denotes the centre
of the dilated ball $\Bik$ and
\begin{align}\label{fe27}
\beta:=\lf(\frac{\ln b}{\ln \lambda_-}+s+1\r)
\frac{\ln \lambda_-}{\ln b}>\frac1{\underline{p}}.
\end{align}
By this, Remark \ref{sr1}(i), Lemma \ref{fl1},
Remark \ref{fr1}(i) and the H\"older inequality,
we find that, for any $t\in(0,\min\{q,1/\bz\})$,
$q_1\in(\frac1{\bz t}, \frac1t)$ and $\delta\in(0,1-\frac1{q_1})$,
\begin{align*}
{\rm J}_{1,2}
&\ls \lf\|
\chi_{\{x\in\rn:\ \frac{2^{k_0\dz}}{(2^{\dz{q_1}'}-1)^{1/{q_1}'}}
\{\sum_{k=-\fz}^{k_0-1}2^{-k\dz{q_1}}
[\sum_{i\in\nn}\lik M(\aik)(x)
\chi_{({A^\tau\Bik})^\com}(x)]^{q_1}\}^
{1/{q_1}}>2^{k_0-2}\}}\r\|_{\lv}\\
&\ls 2^{-k_0q_1(1-\dz)}\lf
\|\sum_{k=-\fz}^{k_0-1}2^{-k\dz{q_1}}\lf
[\sum_{i\in\nn}\lik M(\aik)
\chi_{({A^\tau\Bik})^\com}\r]^{q_1}\r\|_{\lv}\\
&\ls 2^{-k_0q_1(1-\dz)}\lf\{
\sum_{k=-\fz}^{k_0-1}2^{(1-\dz)k{q_1}t}
\lf\|\sum_{i\in\nn}\lf[\HL\lf(\chi_{\Bik}\r)\r]^
{\bz tq_1}\r\|_{L^{\frac{p(\cdot)}
{t}}(\rn)}\r\}^{\frac{1}{t}}\\
&\ls 2^{-k_0q_1(1-\dz)}\lf
[\sum_{k=-\fz}^{k_0-1}2^{(1-\dz)k{q_1}t}
\lf\|\sum_{i\in\nn}\chi_{\Bik}\r\|_
{L^{\frac{p(\cdot)}{t}}(\rn)}\r]^{\frac{1}{t}}\\
&\ls 2^{-k_0q_1(1-\dz)}\lf\{\sum_{k=-\fz}
^{k_0-1}2^{[(1-\dz){q_1}-\varepsilon]kt}
2^{\varepsilon kt}\lf\|\sum_{i\in\nn}
\chi_{A_{j_0}\Bik}\r\|_
{\lv}^{t}\r\}^{\frac{1}{t}}\\
&\ls 2^{-k_0q_1(1-\dz)}\lf\{\sum_{k=-\fz}^{k_0-1}
2^{[(1-\dz){q_1}-\varepsilon]kt\frac q{q-t}}\r\}^{\frac{q-t}{qt}}
\lf[\sum_{k=-\fz}^{k_0-1}2^{\varepsilon kq}\lf\|\sum_{i\in\nn}
\chi_{\Bik}\r\|^{q}_{\lv}\r]^{\frac{1}{q}}\\
&\sim 2^{-\varepsilon k_0}\lf[\sum_{k=-\fz}^{k_0-1}
2^{\varepsilon kq}\lf\|\sum_{i\in\nn}
\chi_{\Bik}\r\|^{q}_{\lv}\r]^{\frac{1}{q}},
\end{align*}
where $\varepsilon\in(1,(1-\dz)q_1)$. This, together with
a calculation similar to that of \eqref{fe8}, further implies that
\begin{align}\label{fe10}
\lf[\sum_{k_0\in\zz}2^{k_0q}({\rm J}_{1,2})^q\r]^{\frac1q}
\ls\lf[\sum_{k\in\zz}2^{kq}
\lf\|\sum_{i\in\nn}\chi_{\Bik}\r\|^{q}_{\lv}\r]^{\frac1q}
\sim\|f\|_{\vahlpq}.
\end{align}
By this, \eqref{fe7} and \eqref{fe8}, we have
\begin{align}\label{fe12}
\lf[\sum_{k_0\in\zz}2^{k_0q}({\rm J}_1)^q\r]^{\frac1q}
\ls\lf[\sum_{k\in\zz}2^{kq}
\lf\|\sum_{i\in\nn}\chi_{\Bik}\r\|^{q}_{\lv}\r]^{\frac1q}
\sim\|f\|_{\vahlpq}.
\end{align}

For ${\rm J}_2$, by Remark \ref{fr1}(i), the H\"older inequality,
we know that,
for any $t\in(0,\min\{\underline{p},q\})$ and $\dz\in(0,1)$,
\begin{align*}
{\rm J}_2&\ls\|\chi_{E_{k_0}}\|_{L^{p(\cdot)}(\rn)}
\ls\lf\|\sum_{k={k_0}}^\fz\sum_{i\in\nn}
\chi_{A^\tau\Bik}\r\|_{L^{p(\cdot)}(\rn)}\\
&\ls\lf\|\sum_{k={k_0}}^\fz\sum_{i\in\nn}
\chi_{\Bik}\r\|_{L^{p(\cdot)}(\rn)}
\ls\lf[\sum_{k={k_0}}^\fz\lf\|\sum_{i\in\nn}\chi_{\Bik}
\r\|_{L^{p(\cdot)}(\rn)}^t\r]^{1/t}\\
&\ls\lf(\sum_{k={k_0}}^\fz2^{-k\dz t\frac q{q-t}}\r)^{\frac{q-t}{qt}}
\lf[\sum_{k={k_0}}^\fz2^{k\dz q}\lf\|\sum_{i\in\nn}\chi_{\Bik}
\r\|_{L^{p(\cdot)}(\rn)}^q\r]^{1/q}\sim2^{-k_0\dz}
\lf[\sum_{k={k_0}}^\fz2^{k\dz q}\lf\|\sum_{i\in\nn}\chi_{\Bik}
\r\|_{L^{p(\cdot)}(\rn)}^q\r]^{1/q}.
\end{align*}
By this, similarly to \eqref{fe8}, we obtain
\begin{align}\label{fe11}
\lf[\sum_{k_0\in\zz}2^{k_0q}({\rm J}_2)^q\r]^{\frac1q}
\ls\lf[\sum_{k\in\zz}2^{kq}
\lf\|\sum_{i\in\nn}\chi_{\Bik}\r\|^{q}_{\lv}\r]^{\frac1q}
\sim\|f\|_{\vahlpq}.
\end{align}

For ${\rm J}_3$, since $\bz>\frac1{\underline{p}}$, it follows that
there exist $0<t<\dz<1$ and $L\in(1,\fz)$ such that $\bz t>\frac1{\underline{p}}$
and $q\bz tL>1$. By this, the value of $\lik$, \eqref{fe9}, Lemma \ref{fl1}
and the H\"older inequality, we find that
\begin{align*}
{\rm J}_3&\ls\lf\|\chi_{\{x\in (E_{k_0})
^\complement:\ \sum_{k=k_0}^\fz\sum_{i\in\nn}
\lik M(\aik)(x)>2^{k_0-1}\}}\r\|_{\lv}\\
&\ls2^{-tk_0}\lf\|\sum_{k=k_0}^\fz
\sum_{i\in\nn}\lf[\lik M(\aik)\r]^t\chi_{(E_{k_0})^\com}\r\|_{\lv}
\ls2^{-tk_0}\lf[\sum_{k=k_0}^\fz
\lf\|\lf\{\sum_{i\in\nn}\lf[\lik M(\aik)\r]^t\chi_{(E_{k_0})^\com}\r\}^
{\frac1{\bz t}}\r\|_{L^{\bz tp(\cdot)}(\rn)}\r]^{\bz t}\\
&\ls2^{-tk_0}\lf[\sum_{k=k_0}^\fz
2^{k/\bz}\lf\|\lf\{\sum_{i\in\nn}\lf[\HL\lf(\chi_{\Bik}\r)\r]^{\bz t}
\r\}^{^{\frac1{\bz t}}}\r\|_{L^{\bz tp(\cdot)}(\rn)}\r]^{\bz t}\\
&\ls2^{-tk_0}\lf[\sum_{k=k_0}^\fz2^{k/\bz}\lf\|\sum_{i\in\nn}
\chi_{\Bik}\r\|_{\lv}^\frac1{\bz t}\r]^{\bz t}
\ls2^{-tk_0}\lf[\sum_{k=k_0}^\fz2^{\frac k{\bz L}(1-\frac \dz t)}
2^{\frac{k\dz}{\bz tL}}\lf\|\sum_{i\in\nn}
\chi_{\Bik}\r\|_{\lv}^\frac1{\bz tL}\r]^{\bz tL}\\
&\ls2^{-tk_0}\lf[\sum_{k=k_0}^\fz2^{\frac k{\bz L}(1-\frac \dz t)
\frac{q\bz tL}{q\bz tL-1}}\r]^{\frac{q\bz tL}q}
\lf[\sum_{k={k_0}}^\fz2^{k\dz q}\lf\|\sum_{i\in\nn}\chi_{\Bik}
\r\|_{L^{p(\cdot)}(\rn)}^q\r]^{1/q}\\
&\sim2^{-\dz k_0}
\lf[\sum_{k={k_0}}^\fz2^{k\dz q}\lf\|\sum_{i\in\nn}\chi_{\Bik}
\r\|_{L^{p(\cdot)}(\rn)}^q\r]^{1/q}.
\end{align*}
From this and an argument similar to that used in the proof of
\eqref{fe8}, we deduce that
 \begin{align}\label{fe28}
\lf[\sum_{k_0\in\zz}2^{k_0q}({\rm J}_3)^q\r]^{\frac1q}
\ls\lf[\sum_{k\in\zz}2^{kq}
\lf\|\sum_{i\in\nn}\chi_{\Bik}\r\|^{q}_{\lv}\r]^{\frac1q}
\sim\|f\|_{\vahlpq},
\end{align}
which, combined with \eqref{fe6}, \eqref{fe12} and \eqref{fe11},
implies that
\begin{align*}
\|f\|_{\vhlpq}&\sim\lf[\sum_{k_0\in\zz}2^{k_0q}
\lf\|\chi_{\{x\in\rn:\ M(f)(x)>2^{k_0}\}}\r\|_{\lv}^q\r]^{\frac1q}\\
&\ls\lf[\sum_{k_0\in\zz}2^{k_0q}({\rm J}_1)^q\r]^{\frac1q}+
\lf[\sum_{k_0\in\zz}2^{k_0q}({\rm J}_2)^q\r]^{\frac1q}+
\lf[\sum_{k_0\in\zz}2^{k_0q}({\rm J}_3)^q\r]^{\frac1q}\ls\|f\|_{\vahlpq}.
\end{align*}
This further implies that $f\in\vhlpq$ and hence finishes the proof of
\eqref{fe4}.

We now prove that $\vhlpq\st\vahlpq$. To this end, it suffices
to show that
\begin{align}\label{fe13}
\vhlpq\st H_A^{p(\cdot),\fz,s,q}(\rn),
\end{align}
due to the fact that each $(p(\cdot),\infty,s)$-atom is also a $(p(\cdot),r,s)$-atom
and hence $$H_A^{p(\cdot),\fz,s,q}(\rn)\st\vahlpq.$$

Now we prove \eqref{fe13} by three steps.

\emph{Step 1.}
To show \eqref{fe13}, for any
$f\in\vhlpq,$ $\phi\in\cs(\rn)$ with $\int_\rn\phi(x)\,dx=1$
and $\nu\in\mathbb{N}$,
let $f^{(\nu)}:=f\ast\phi_{-\nu}$. By this and \cite[p.\,15, Lemma 3.8]{mb03},
we know that $f^{(\nu)}\to f$ \ in $\cs'(\rn)$ as $\nu\to\infty$.
Moreover, by \cite[p.\,39, Lemma 6.6]{mb03}, we conclude that there exists
a positive constant $ C_{(N,\phi)}$, depending
on $N$ and $\phi$, but independent of $f$, such that,
for any $\nu\in\mathbb{N}$ and $x\in\rn$,
\begin{align*}
M_{N+2}\lf(f^{(\nu)}\r)(x)\le C_{(N,\phi)} M_N(f)(x).
\end{align*}
Thus, by Definition \ref{sd3} and Proposition \ref{sp1},
we find that $f^{(\nu)}\in\vhlpq$
and $$\lf\|f^{(\nu)}\r\|_{\vhlpq}\ls\|f\|_{\vhlpq}.$$

The aim of this step is to prove that, for any $\nu\in\nn$,
\begin{align}\label{fe15}
f^{(\nu)}=\sum_{k\in\zz}\sum_{i\in\nn}h_i^{\nu,k}
\hspace{0.6cm} {\rm in}\hspace{0.2cm}\cs'(\rn),
\end{align}
where, for each $\nu,\,i\in\nn$ and $k\in\zz$, $h_i^{\nu,k}$
is a $(p(\cdot),\fz,s)$-atom multiplied by a constant
depending on $k$ and $i$, but independent of $f$ and $\nu$.

To show \eqref{fe15}, we borrow some ideas from the
proofs of \cite[p.\,38, Theorem 6.4]{mb03} and \cite[Theorem 3.6]{lyy16}.
For any $k\in\mathbb{Z}$ and
$N\in\mathbb{N}\cap[\lfloor(\frac1{\underline{p}}-1)\frac{\ln b}{\ln
\lambda_-}\rfloor+2,\fz)$,
let
$$\Omega_k:=\lf\{x\in\rn:\ M_N(f)(x)>2^k\r\}.$$
Clearly, $\Omega_k$ is open. From this and Lemma \ref{fl3} with $m=6\tau$,
we deduce that there exist a sequence
$\{x_i^k\}_{i\in\mathbb{N}}\subset\Omega_k$ and
a sequence $\{\ell_i^k\}_{i\in\mathbb{N}}\st\zz$
such that
\begin{align}\label{fe16}
\Omega_k=\bigcup_{i\in\mathbb{N}}(x_i^k+B_{\ell_i^k});
\end{align}
\begin{align}\label{fe17}
(x_i^k+B_{\ell_i^k-\tau})\cap(x_j^k+B_{\ell_j^k-\tau})
=\emptyset\hspace{0.2cm} {\rm for\ any}\  i,\ j\in\nn\ {\rm with}\ i\neq j;
\end{align}
\begin{align*}
(x_i^k+B_{\ell_i^k+6\tau})\cap\Omega_k^\complement
=\emptyset,\hspace{0.2cm} (x_i^k+B_{\ell_i^k+6\tau+1})\cap
\Omega_k^\complement\neq\emptyset\ \ \ {\rm for\ any}\
i\in\mathbb{N};
\end{align*}
\begin{align*}
{\rm if}\ (x_i^k+B_{\ell_i^k+4\tau})\cap(x_j^k+B_{\ell_j^k+
4\tau})\neq\emptyset,\hspace{0.1cm} {\rm then}\ |\ell_i^k-\ell_j^
k|\le\tau;
\end{align*}
\begin{align}\label{fe18}
\sharp\lf\{j\in\mathbb{N}:\
(x_i^k+B_{\ell_i^k+4\tau})\cap(x_j^k+B_{\ell_j^k+
4\tau})\neq\emptyset\r\}\le R\hspace{0.2cm} {\rm for\ any}\  i\in\mathbb{N},
\end{align}
where $\tau$ and $R$ are same as in Lemma \ref{fl3}.

Fix $\eta\in\cs(\rn)$ satisfying $\supp\eta\subset B_\tau$,
$0\le\eta\le1$ and $\eta\equiv1$ on $B_0$. For
any $i\in\mathbb{N}$, $k\in\mathbb{Z}$ and $x\in\rn$,
let
$\eta_i^k(x):=\eta(A^{-\ell_i^k}(x-x_i^k))$
and
$$\theta_i^k(x):=
\frac {\eta_i^k(x)}{\Sigma_{j\in\nn}\eta_j^k(x)}.$$
Then $\theta_i^k\in\cs(\rn)$,
$\supp \theta_i^k\subset x_i^k+B_{\ell_i^k+\tau}$,
$0\le\theta_i^k\le1$, $\theta_i^k\equiv1$
on $x_i^k+B_{\ell_i^k-\tau}$ by \eqref{fe17}, and
$\sum_{i\in\mathbb{N}}\theta_i^k=\chi_{\Omega_k}$.
From this, it follows that $\{\theta_i^k\}_{i\in\mathbb{N}}$
forms a smooth partition of unity of $\Omega_k$.

For any $m\in\mathbb{Z}_+$, define $\mathcal{P}_m(\rn)$ to be
the linear space of all polynomials on $\rn$ with
degree not greater than $m$. Moreover, for each $i$ and
$P\in\mathcal{P}_m(\rn)$, let
\begin{align}\label{fe20}
\|P\|_{i,k}:=
\lf[\frac1{\int_\rn\theta_i^k(x)\,dx}\int_\rn
|P(x)|^2\theta_i^k(x)\,dx\r]^{1/2}
\end{align}
Then $(\mathcal{P}_m(\rn), \|\cdot\|_{i,k})$
is a finite dimensional Hilbert space. For each $i$, since
$f^{(\nu)}$ induces a linear functional on
$\mathcal{P}_{m}(\rn)$ via
$$Q\mapsto\frac1{\int_\rn\theta_i^k(x)\,dx}\lf\langle
f^{(\nu)},Q\theta_i^k\r\rangle,\hspace{0.2cm} Q\in\mathcal{P}_m(\rn),$$
by the Riesz lemma, it follows that there exists
a unique polynomial  $P_i^{\nu,k}\in\mathcal{P}_m(\rn)$
such that, for any $Q\in\mathcal{P}_m(\rn)$,
\begin{align*}
\frac1{\int_\rn\theta_i^k(x)\,dx}
\lf\langle f^{(\nu)},Q\theta_i^k\r\rangle
&=\frac1{\int_\rn\theta_i^k(x)\,dx}
\lf\langle P_i^{\nu,k},Q\theta_i^k\r\rangle\\
&=\frac 1{\int_\rn\theta_i^k(x)\,dx}
\int_\rn P_i^{\nu,k}(x)Q(x)\theta_i^k(x)\,dx.
\end{align*}
For each $i\in\mathbb{N}$ , $k\in\mathbb{Z}$ and
$\nu\in\mathbb{N}$, let
$b_i^{\nu,k}:=[f^{(\nu)}-P_i^{\nu,k}]\theta_i^k$
and
$$g^{\nu,k}:=f^{(\nu)}-\sum_{i\in\mathbb{N}}b_i^{\nu,k}
=f^{(\nu)}-\sum_{i\in\mathbb{N}}\lf[f^{(\nu)}-P_i^{\nu,k}\r]
\theta_i^k =f^{(\nu)}\chi_{\Omega_k^\complement}+
\sum_{i\in\mathbb{N}}P_i^{\nu,k}\theta_i^k.$$
Then, by an argument similar to that used in \cite[p.\,1679]{lyy16},
we conclude that, for any $\nu\in\nn$,
$\|g^{\nu,k}\|_{L^{\infty}(\rn)}\ls2^k$ and
$\|g^{\nu,k}\|_{L^{\infty}(\rn)}\to0$
as $k\to-\infty$.

For any $k\in\mathbb{Z}$, let
$$\widetilde{\Omega}_k:=\lf\{x\in\rn:\ M_N(f^{(\nu)})(x)>2^k\r\}.$$
Then, for any $\epsilon\in(0,1)$, by $f^{(\nu)}\in\vhlpq$,
we know that there exists an integer $k_{(\epsilon)}$ such that,
for any $k\in[k_{(\epsilon)},\fz)\cap\zz$, $|\wz{\Omega}_k|<\epsilon$.
Since, for any $k\in[k_{(\epsilon)},\fz)\cap\zz$ and $\alpha\in(0,\fz)$,
\begin{align*}
\lf|\widetilde{\Omega}_k\cap\lf\{x\in\rn:\ M_N(f^{(\nu)})(x)>\alpha\r\}\r|
&\le\min\lf\{\lf|\widetilde{\Omega}_k\r|,\ \alpha^{-p_-}
\lf\|M_N(f^{(\nu)})\r\|_{L^{p(\cdot),\infty}(\widetilde{\Omega}_k)}^{p_-}\r\}\\
&\le\min\lf\{\lf|\widetilde{\Omega}_k\r|,\ \alpha^{-p_-}
\lf\|M_N(f^{(\nu)})\r\|_{L^{p(\cdot),\infty}(\rn)}^{p_-}\r\},
\end{align*}
it follows that, for any $p_0\in(0,p_-)$ satisfying
$\lfloor(1/p_0-1)\ln b/\ln\lambda_-\rfloor\le s$,
\begin{align}\label{fe19}
&\int_{\widetilde{\Omega}_k}\lf[M_N(f^{(\nu)})(x)\r]^{p_0}\,dx\\
&\hs=\int_0^\infty p_0\alpha^{p_0-1}\lf|\lf\{x\in \widetilde{\Omega}_k:\
M_N(f^{(\nu)})(x)>\alpha\r\}\r|\,d\alpha\noz\\
&\hs\le\int_0^\gamma p_0\lf|\widetilde{\Omega}_k\r|
\alpha^{p_0-1}\,d\alpha+\int_\gamma^\infty p_0\alpha^
{p_0-1-p_-}\lf\|M_N(f^{(\nu)})\r\|_{L^{p(\cdot),\infty}(\rn)}^{p_-}\,d\alpha\noz\\
&\hs=\frac {p_-}{p_--p_0}\lf|\widetilde{\Omega}_k\r|^{1-\frac{p_0}{p_-}}
\lf\|M_N(f^{(\nu)})\r\|_{L^{p(\cdot),\infty}(\rn)}^{p_0}
\ls\lf|\widetilde{\Omega}_k\r|^{1-\frac{p_0}{p_-}}
\lf\|M_N(f^{(\nu)})\r\|_{L^{p(\cdot),q}(\rn)}^{p_0},\noz
\end{align}
where
$\gamma:={\|M_N(f^{(\nu)})\|_{L^{p(\cdot),\infty}
(\rn)}}{|\widetilde{\Omega}_k|^{-\frac1{p_-}}}$.
On the other hand, by \cite[(3.9)]{lyy16}, we know that, for any
$k\in\mathbb{Z}$ and $p_0$ as in \eqref{fe19},
\begin{align*}
\int_\rn \lf[M_N\lf(\sum_{i\in\mathbb{N}}
b_i^{\nu,k}\r)(x)\r]^{p_0}\,dx
\ls\int_{\widetilde{\Omega}_k}
\lf[M_N(f^{(\nu)})(x)\r]^{p_0}\,dx,
\end{align*}
which, together with \eqref{fe19}, further implies that
\begin{align*}
\lf\|\sum_{i\in\mathbb{N}}b_i^{\nu,k}\r\|_{H^{p_0}_A(\rn)}
:=&\lf\|M_N\lf(\sum_{i\in\mathbb{N}}b_i^{\nu,k}\r)\r\|_
{L^{p_0}(\rn)}\ls\lf|\widetilde{\Omega}_k\r|^{\frac1{p_0}-\frac1{p_-}}
\lf\|M_N(f^{(\nu)})\r\|_{L^{p(\cdot),q}(\rn)}\to0
\end{align*}
as $k\to\infty$.
Here and hereafter, for any $p\in(0,\fz)$,
$H^p_A(\rn)$ denotes the anisotropic Hardy space
introduced by Bownik in \cite{mb03}. By this and the fact that,
for any $\nu\in\nn$,
$\|g^{\nu,k}\|_{L^{\infty}(\rn)}\to0$
as $k\to-\infty$, we have
\begin{align*}
&\lf\|f^{(\nu)}-\sum_{k=-N}^N\lf(g^{\nu,k+1}-g^{\nu,k}\r)\r\|_
{H^{p_0}_A(\rn)+L^\infty(\rn)}\\
&\hs\ls\lf\|\sum_{i\in\nn}
b_i^{\nu,N+1}\r\|_{H^{p_0}_A(\rn)}+\lf\|g^{\nu,-N}\r\|_
{L^\infty(\rn)}\to0\hspace{0.2cm} {\rm as}
\hspace{0.2cm}N\to\infty,
\end{align*}
where, for any $f\in H^{p_0}_A(\rn)+L^\infty(\rn)$,
\begin{align*}
\|f\|_{H^{p_0}_A(\rn)+L^\infty(\rn)}:=\inf
\Big\{\|&f_1\|_{H^{p_0}_A(\rn)}+\|f_2\|_{L^\infty(\rn)}:\ \\
&\lf.f=f_1+f_2,\ f_1\in H^{p_0}_A(\rn),\ f_2\in L^\infty(\rn)\r\}
\end{align*}
with the infimum being taken over all decompositions of $f$
as above. From this and a proof similar to the proofs of
\cite[(3.12), (3.13), (3.14) and (3.16)]{lyy16}, we deduce that,
for any $\nu\in\nn$,
\begin{align*}
f^{(\nu)}&=\sum_{k\in\zz}\lf(g^
{\nu,k+1}-g^{\nu,k}\r)\\
&=\sum_{k\in\zz}\sum_{i\in\nn}\lf[b_i^{\nu,k}-\sum_{j\in\nn}
\lf(b_j^{\nu,k+1}\theta_i^k-P_{i,j}^{\nu,k+1}\theta_j^
{k+1}\r)\r]=:\sum_{k\in\zz}\sum_{i\in\zz}h_i^{\nu,k}
\hspace{0.3cm} {\rm in}\hspace{0.2cm} \cs'(\rn),
\end{align*}
where, for any $\nu$, $i$, $j\in\nn$ and $k\in\zz$,
$P_{i,j}^{\nu,k+1}$ is the orthogonal projection of
$[f^{(\nu)}-P_j^{\nu,k+1}]\theta_i^k$ on $\mathcal{P}_{m}(\rn)$
with respect to the norm defined by \eqref{fe20} and $h_i^{\nu,k}$
is a multiple of a $(p(\cdot),\fz,s)$-atom satisfying
\begin{align}\label{fe21}
\int_\rn h_i^{\nu,k}(x)Q(x)\,dx=0\hspace{0.4cm}
{\rm for\ any}\ Q\in\mathcal{P}_{m}(\rn),
\end{align}
\begin{align}\label{fe22}
\supp h_i^{\nu,k}\subset (x_i^k+B_{\ell_i^k+4\tau})
\end{align}
and
\begin{align}\label{fe23}
\lf\|h_i^{\nu,k}\r\|_{L^{\infty}(\rn)}\ls2^k.
\end{align}
This finishes the proof of \eqref{fe15}.

\emph{Step 2.} From \eqref{fe23} and the Alaoglu theorem (see, for example,
\cite[Theorem 3.17]{wr91}), it follows that there exists a subsequence
$\{\nu_\iota\}_{\iota=1}^{\fz}\st\nn$ such that, for each $i\in\nn$ and $k\in\zz$,
$h_i^{\nu_\iota,k}\to h_i^k$ as $\iota\to\fz$ in the weak-$\ast$ topology of
$L^{\fz}(\rn)$. Moreover,
$\supp h_i^k\subset (x_i^k+B_{\ell_i^k+4\tau})$,
$\|h_i^k\|_{L^{\infty}(\rn)}\ls2^k$ and
$\int_\rn h_i^k(x)Q(x)dx=0$ for any $Q\in\mathcal{P}_{m}(\rn)$.
Therefore, $h_i^k$ is a multiple of a $(p(\cdot),\infty,s)$-atom $a_i^k$.
Let $h_i^k:=\lambda_i^ka_i^k$, where
$\lambda_i^k\sim 2^k\|\chi_{\xik+B_{\ell_i^k+4\tau}}\|_{\lv}$. Then,
by \eqref{fe16}, \eqref{fe18} and Lemma \ref{sl3*}, we have
\begin{align*}
&\lf[\sum_{k\in\zz}\lf\|\lf\{\sum_{i\in\nn}
\lf[\frac{\lz_i^k\chi_{\xik+B_{\ell_i^k+4\tau}}}
{\|\chi_{\xik+B_{\ell_i^k+4\tau}}\|_{\lv}}\r]^
{\underline{p}}\r\}^{1/\underline{p}}\r\|_{\lv}^q\r]^{\frac1q}\\
&\hs\sim\lf[\sum_{k\in\zz}2^{kq}\lf\|\lf(\sum_{i\in\nn}
\chi_{\xik+B_{\ell_i^k+4\tau}}\r)^{\frac 1{p_-}}\r\|_{\lv}^q\r]^{\frac1q}\\
&\hs\sim\lf(\sum_{k\in\zz}2^{kq}\lf\|\chi_{\Omega_k}\r\|_{\lv}^q\r)^{\frac1q}
\sim\|M_N(f)\|_{L^{p(\cdot),q}(\rn)}\sim\|f\|_{\vhlpq}
\end{align*}
with the usual modification made when $q=\fz$.

\emph{Step 3.} By Step 2, we conclude that, to
prove \eqref{fe13}, it suffices to show that
$f=\sum_{k\in\zz}\sum_{i\in\nn}h_i^k$ in $\cs'(\rn)$.

To this end, for any $k\in\zz$, let
$$f_k:=\sum_{i\in\nn}h_i^k$$
and, for any $\nu\in\nn$ and $k\in\zz$,
$f^{(\nu)}_k:=g^{\nu,k+1}-g^{\nu,k}$. Then
$f_k^{(\nu_\iota)}\to f_k$ in $\cs'(\rn)$ as
$\iota\to\fz$. Indeed, by \eqref{fe18} and the support
conditions of $h_i^k$ and $h_i^{\nu,k}$, we find that,
for any $\varphi\in\cs(\rn)$,
\begin{align*}
\lf\langle f_k^{(\nu_\iota)},\varphi\r\rangle
&=\lf\langle \sum_{i\in\mathbb{N}}h_i^
{\nu_\iota,k},\varphi\r\rangle=\sum_{i\in\nn}
\lf\langle h_i^{\nu_\iota,k},\varphi\r\rangle\\
&\to\sum_{i\in\mathbb{N}}
\lf\langle h_i^k,\varphi\r\rangle=\lf\langle \sum_
{i\in\nn}h_i^k,\varphi\r\rangle=\langle f_k,
\varphi\rangle\hspace{0.2cm}{\rm as}\ \iota\to\fz.
\end{align*}

Next we aim to prove that, for any $\nu\in\nn$,
\begin{align}\label{fe26}
\sum_{|k|\geq K_1}f^{(\nu)}_k\to0\hspace{0.25cm}
{\rm in}\hspace{0.2cm} \cs'(\rn)\hspace{0.2cm} {\rm as}\ K_1\to\fz.
\end{align}
Indeed, for any $p_0\in(0,p_-)$ with
$\lfloor(1/p_0-1)\ln b/\ln\lambda_-\rfloor\le s$, by \eqref{fe21}, \eqref{fe22}
and \eqref{fe23}, we know that
$(2^k|B_{\ell_i^k+4\tau}|^{1/p_0})^{-1}h_i^{\nu,k}$
is a $(p_0,\fz,s)$-atom multiplied by a constant.
From this, \cite[p.\,19, Theorem 4.2]{mb03}, \eqref{fe18},
\eqref{fe16} and Lemma \ref{sl3*}, we deduce that,
as $K_2\to\fz$,
\begin{align}\label{fe24}
\lf\|\sum_{k\geq K_2}f^{(\nu)}_k\r\|_{H^{p_0}_A(\rn)}^{p_0}&\le\sum_{k\geq K_2}\sum_{i\in\nn}
\lf\|h_i^{\nu,k}\r\|_{H^{p_0}_A(\rn)}^{p_0}
\ls\sum_{k\geq K_2}\sum_{i\in\nn}2^{kp_0}
\lf|B_{\ell_i^k+4\tau}\r|\\
&\ls\sum_{k\geq K_2}2^{kp_0}
\lf[\lf\|\chi_{\Omega_k}\r\|_{\lv}^{p_-}
+\lf\|\chi_{\Omega_k}\r\|_{\lv}^{p_+}\r]\noz\\
&\ls\sum_{k\geq K_2}\lf[2^{k(p_0-p_-)}
\|f\|_{H^{p(\cdot),\fz}_A(\rn)}^{p_-}+2^{k(p_0-p_+)}
\|f\|_{H^{p(\cdot),\fz}_A(\rn)}^{p_+}\r]\noz\\
&\ls\sum_{k\geq K_2}\lf[2^{k(p_0-p_-)}
\|f\|_{\vhlpq}^{p_-}+2^{k(p_0-p_+)}\|f\|_{\vhlpq}^{p_+}\r]
\to0.\noz
\end{align}
Let
$E_1:=\{x\in\rn:\ p(x)\in(0,1)\}$, $E_2:=\{x\in\rn:\ p(x)\in[1,\fz)\}$
and, for any $j\in\{1,2\}$,
\begin{align*}
p_-(E_j):=\mathop\mathrm{ess\,inf}_{x\in E_j}p(x)\hspace{0.35cm}
{\rm and}\hspace{0.35cm}
p_+(E_j):=\mathop\mathrm{ess\,sup}_{x\in E_j}p(x).
\end{align*}
Then, by \eqref{fe21}, \eqref{fe22}
and \eqref{fe23} again, we find that
$(2^k|B_{\ell_i^k+4\tau}|)^{-1}h_i^{\nu,k}$ is a
$(1,\fz,s)$-atom multiplied by a constant. Therefore,
by \cite[p.\,19, Theorem 4.2]{mb03}, \eqref{fe18},
\eqref{fe16} and Lemma \ref{sl3*} again,
we conclude that
\begin{align}\label{fe25}
\lf\|\sum_{k\le K_3}f^{(\nu)}_k\r\|_{H^1_A(E_1)}&\le\sum_{k\le K_3}
\lf\|\sum_{i\in\nn}h_i^{\nu,k}\chi_{E_1}\r\|_{H^1_A(\rn)}
\ls\sum_{k\le K_3}2^k\lf|\Omega_k\cap E_1\r|\\
&\ls\sum_{k\le K_3}2^k\lf[\lf\|\chi_{\Omega_k}\r\|_{\lv}^{p_-(E_1)}
+\lf\|\chi_{\Omega_k}\r\|_{\lv}^{p_+(E_1)}\r]\noz\\
&\ls\sum_{k\le K_3}\lf[2^{k(1-p_-(E_1))}
\|f\|_{H^{p(\cdot),\fz}_A(\rn)}^{p_-(E_1)}+2^{k(1-p_+(E_1))}
\|f\|_{H^{p(\cdot),\fz}_A(\rn)}^{p_+(E_1)}\r]\noz\\
&\ls\sum_{k\le K_3}\lf[2^{k(1-p_-(E_1))}
\|f\|_{\vhlpq}^{p_-(E_1)}+2^{k(1-p_+(E_1))}\|f\|_{\vhlpq}^{p_+(E_1)}\r]
\to0\noz
\end{align}
as $K_3\to-\fz$. Similarly, for any $\wz{p}_+\in(p_+(E_2),\fz)$,
we have
\begin{align*}
\lf\|\sum_{k\le K_4}f^{(\nu)}_k\r\|_{L^{\wz{p}_+}(E_2)}
&\le\sum_{k\le K_4}
\lf\|\sum_{i\in\nn}h_i^{\nu,k}\chi_{E_2}\r\|_{L^{\wz{p}_+}(\rn)}
\ls\sum_{k\le K_4}
2^k\lf|\Omega_k\cap E_2\r|^{\frac1{\wz{p}_+}}\\
&\ls\sum_{k\le K_4}2^k\lf[\lf\|\chi_{\Omega_k}\r\|_{\lv}
^{\frac{p_-(E_2)}{\wz{p}_+}}
+\lf\|\chi_{\Omega_k}\r\|_{\lv}
^{\frac{p_+(E_2)}{\wz{p}_+}}\r]\\
&\ls\sum_{k\le K_4}\lf[2^{k(1-\frac{p_-(E_2)}{\wz{p}_+})}
\|f\|_{H^{p(\cdot),\fz}_A(\rn)}^{\frac{p_-(E_2)}{\wz{p}_+}}
+2^{k(1-\frac{p_+(E_2)}{\wz{p}_+})}
\|f\|_{H^{p(\cdot),\fz}_A(\rn)}^{\frac{p_+(E_2)}{\wz{p}_+}}\r]\\
&\ls\sum_{k\le K_4}\lf[2^{k(1-\frac{p_-(E_2)}{\wz{p}_+})}
\|f\|_{\vhlpq}^{\frac{p_-(E_2)}{\wz{p}_+}}
+2^{k(1-\frac{p_+(E_2)}{\wz{p}_+})}
\|f\|_{\vhlpq}^{\frac{p_+(E_2)}{\wz{p}_+}}\r]
\to0\noz
\end{align*}
as $K_4\to-\fz$. This, combined with \eqref{fe24} and
\eqref{fe25}, implies that \eqref{fe26} holds true.

By an argument
similar to that used in the proof of \eqref{fe26}, we know that
$\sum_{|k|\geq K_1}f_k\to0$ in $\cs'(\rn)$
as $K_1\to\fz$. From this, \eqref{fe26} and
a proof similar to that used in \cite[pp.\,1682-1683]{lyy16},
we further deduce that
$$f=\sum_{k\in\mathbb{Z}}f_k=\sum_{k\in\zz}\sum_{i\in\nn}h_i^k\quad
\mathrm{in}\quad \cs'(\rn),$$
which completes the proof of \eqref{fe13}. This shows $\vhlpq\subset \vahlpq$
and hence finishes the proof of Theorem \ref{ft1}.
\end{proof}

\section{Lusin area function characterizations of $\vhlpq$\label{s5}}
\hskip\parindent
In this section, using the atomic
characterization of $\vhlpq$ obtained in Theorem \ref{ft1},
we establish the Lusin area
function characterization of $\vhlpq$ in Theorem
\ref{fivet1}.
To this end, we first recall the notion of the anisotropic
Lusin area function (see \cite{lfy14,lyy16}).

\begin{definition}\label{fived1}
Let $\psi\in\cs(\rn)$ satisfy
$\int_{\rn}\psi(x)x^\gamma\,dx=0$
for any multi-index $\gamma\in(\zz_+)^n$ with
$|\gamma|\leq s$, where
$s\in\nn\cap[\lfloor(\frac1{p_-}-1)\ln b/\ln\lambda_-\rfloor,\fz)$ and
$p_-$ is as in \eqref{se3}. Then, for any $f\in\cs'(\rn)$,
the \emph{anisotropic Lusin area function} $S(f)$
is defined by setting, for any $x\in\rn$,

\begin{align*}
S(f)(x):=\lf[\sum_{k\in\mathbb{Z}}b^{-k}\int_{x+B_k}
\lf|f\ast\psi_{k}(y)\r|^2\,dy\r]^{1/2},
\end{align*}
\end{definition}

Recall also that a distribution $f\in\cs'(\rn)$ is said to
\emph{vanish weakly at infinity} if, for each $\psi\in\cs(\rn)$,
$f\ast\psi_{k}\to0$ in $\cs'(\rn)$ as $k\to\fz$.
Denote by $\cs'_0(\rn)$ the \emph{set of all $f\in\cs'(\rn)$
vanishing weakly at infinity}.

The main result of this section is the following Theorem \ref{fivet1}.
\begin{theorem}\label{fivet1}
Let $p(\cdot)\in C^{\log}(\rn)$ and $q\in(0,\fz]$.
Then $f\in\vhlpq$ if and only if
$f\in\cs'_0(\rn)$ and $S(f)\in\vlpq$.
Moreover, there exists a positive constant $C$ such that,
for any $f\in\vhlpq$,
$$C^{-1}\|S(f)\|_{\vlpq}\le\|f\|_{\vhlpq}\le C\|S(f)\|_{\vlpq}.$$
\end{theorem}

To prove Theorem \ref{fivet1},
we need several technical lemmas. First, by an argument similar
to that used in the proof of \cite[Lemma 6.5]{yyyz16}, it is easy
to see that the following lemma holds true, the details being omitted.

\begin{lemma}\label{fivel1}
Let $p(\cdot)\in C^{\log}(\rn)$ and $q\in(0,\fz]$.
Then
$\vhlpq\subset\cs'_0(\rn)$.
\end{lemma}

Via borrowing some ideas from \cite[Lemma 2.6]{zyl16}
and \cite[Lemma 2.2]{ns12},
we obtain the following result,
which is an anisotropic version of
\cite[Lemma 2.6]{zyl16}.

\begin{lemma}\label{fivel4}
Let $p(\cdot)\in C^{\log}(\rn)$ and $p_-\in(1,\fz)$.
Then there exists a positive constant $C$ such that,
for all subsets $E_1$, $E_2$ of $\rn$ with $E_1\st E_2$,
\begin{align*}
C^{-1}\lf(\frac{|E_1|}{|E_2|}\r)^{\frac 1{p_-}}
\le\frac{\|\chi_{E_1}\|_{L^{p(\cdot)}(\rn)}}
{\|\chi_{E_2}\|_{L^{p(\cdot)}(\rn)}}
\le C\lf(\frac{|E_1|}{|E_2|}\r)^{\frac 1{p_+}}.
\end{align*}
\end{lemma}

\begin{proof}
In view of similarity, we only show that
\begin{align}\label{five13}
\frac{\|\chi_{E_1}\|_{L^{p(\cdot)}(\rn)}}
{\|\chi_{E_2}\|_{L^{p(\cdot)}(\rn)}}
\ls\lf(\frac{|E_1|}{|E_2|}\r)^{\frac 1{p_+}}.
\end{align}
To this end, for any $i\in\{1,2\}$, let
\begin{align*}
p_-(E_i):=\mathop\mathrm{ess\,inf}_{x\in E_i}p(x)\hspace{0.35cm}
{\rm and}\hspace{0.35cm}
p_+(E_i):=\mathop\mathrm{ess\,sup}_{x\in E_i}p(x).
\end{align*}

If $|E_2|\le1$, then, by \eqref{se6} and $p_-\in(1,\fz)$, we know that,
for any $i\in\{1,2\}$ and $x\in E_i$,
\begin{align*}
|E_i|^{\frac1{p(x)}}
\sim|E_i|^{\frac1{p_-(E_i)}}\sim|E_i|^{\frac1{p_+(E_i)}},
\end{align*}
which, combined with
\begin{align*}
|E_i|^{\frac1{p_-(E_i)}}
\le\lf\|\chi_{E_i}\r\|_\lv\le|E_i|^{\frac1{p_+(E_i)}},
\end{align*}
implies that
\begin{align}\label{five14}
\lf\|\chi_{E_i}\r\|_\lv\sim|E_i|^{\frac1{p(x)}}
\sim|E_i|^{\frac1{p_-(E_i)}}\sim|E_i|^{\frac1{p_+(E_i)}}.
\end{align}
By this, we conclude that, for any $x\in E_1$,
\begin{align}\label{five15}
\frac{\|\chi_{E_1}\|_{L^{p(\cdot)}(\rn)}}
{\|\chi_{E_2}\|_{L^{p(\cdot)}(\rn)}}
\sim\lf(\frac{|E_1|}{|E_2|}\r)^{\frac1{p(x)}}
\ls\lf(\frac{|E_1|}{|E_2|}\r)^{\frac 1{p_+}}.
\end{align}

If $|E_1|\ge1$, let $\{Q_j\}_{j\in\nn}$ be a partition of
$\rn$ such that, for any $i$, $j\in\nn$, $|Q_i|=|Q_j|=1$ and,
when ${\rm dist}(\vec{0}_n,Q_i)>{\rm dist}(\vec{0}_n,Q_j)$, $i>j$,
where, for any $i\in\nn$,
${\rm dist}(\vec{0}_n,Q_i):=\inf\{|x|:\ x\in Q_i\}$.
Then, by \cite[Theorem 2.4]{h09} and \eqref{five14}, we find that,
for any $i\in\{1,2\}$,
\begin{align}\label{five16}
\lf\|\chi_{E_i}\r\|_\lv\sim\lf\|\lf\{\lf\|
\chi_{E_i}\r\|_{L^{p(\cdot)}(Q_j)}\r\}_{j\in\nn}\r\|_{\ell^{p_\fz}}
\sim|E_i|^{\frac1{p_\fz}},
\end{align}
where $p_\fz$ is as in \eqref{se7}. From this, it follows that
\begin{align}\label{five17}
\frac{\|\chi_{E_1}\|_{L^{p(\cdot)}(\rn)}}
{\|\chi_{E_2}\|_{L^{p(\cdot)}(\rn)}}
\sim\lf(\frac{|E_1|}{|E_2|}\r)^{\frac1{p_\fz}}
\ls\lf(\frac{|E_1|}{|E_2|}\r)^{\frac 1{p_+}}.
\end{align}

If $|E_1|<1<|E_2|$, then, from \eqref{five14} and \eqref{five16},
we deduce that
\begin{align}\label{five18}
\frac{\|\chi_{E_1}\|_{L^{p(\cdot)}(\rn)}}
{\|\chi_{E_2}\|_{L^{p(\cdot)}(\rn)}}
\sim\frac{|E_1|^{\frac1{p_+(E_1)}}}{|E_2|^{\frac1{p_\fz}}}
\ls\lf(\frac{|E_1|}{|E_2|}\r)^{\frac 1{p_+}}.
\end{align}
Combining \eqref{five15}, \eqref{five17} and \eqref{five18},
we obtain \eqref{five13}. This finishes the proof of Lemma \ref{fivel4}.
\end{proof}

The following lemma is just \cite[Lemma 2.3]{blyz10},
which is a slight modification of \cite[Theorem 11]{cm90}.

\begin{lemma}\label{fivel2}
Let $A$ be a dilation. Then there exists a collection
$$\mathcal{Q}:=\lf\{Q_\alpha^k\subset\rn:\ k\in\mathbb{Z},
\,\alpha\in I_k\r\}$$
of open subsets, where $I_k$ is certain index set, such that
\begin{enumerate}
\item[{\rm (i)}] for any $k\in\zz$,
$\lf|\rn\setminus\bigcup_{\alpha}Q_\alpha^k\r|=0$
and, when $\alpha\neq\beta$,
$Q_\alpha^k\cap Q_\beta^k=\emptyset$;
\item[{\rm(ii)}] for any $\alpha,\,\beta,\,k,\,\ell$ with $\ell\geq k$,
either $Q_\alpha^k\cap Q_\beta^\ell=\emptyset$ or
$Q_\alpha^\ell\subset Q_\beta^k$;
\item[{\rm(iii)}] for each $(\ell,\beta)$ and each $k<\ell$,
there exists a unique $\alpha$ such that
$Q_\beta^\ell\subset Q_\alpha^k$;
\item[{\rm(iv)}] there exist some negative integer $v$
and positive integer $u$ such that, for any $Q_\alpha^k$
with $k\in\mathbb{Z}$ and $\alpha\in I_k$,
there exists $x_{Q_\alpha^k}\in Q_\alpha^k$
satisfying that, for any $x\in Q_\alpha^k$,
$$x_{Q_\alpha^k}+B_{vk-u}
\subset Q_\alpha^k\subset x+B_{vk+u}.$$
\end{enumerate}
\end{lemma}

In what follows, we call
$\mathcal{Q}:=
\{Q_\alpha^k\}_{k\in\mathbb{Z},\,\alpha\in I_k}$
from Lemma \ref{fivel2} \emph{dyadic cubes} and
$k$ the \emph{level}, denoted by $\ell(Q_\alpha^k)$,
of the dyadic cube $Q_\alpha^k$
with $k\in\mathbb{Z}$ and $\alpha\in I_k$.

\begin{remark}\label{fiver1}
In the definition of $(p(\cdot),r,s)$-atoms (see Definition \ref{fd1}),
if we replace dilated balls $\mathfrak{B}$ by
dyadic cubes, then, from Lemma \ref{fivel2}, we deduce that
the corresponding anisotropic variable atomic Hardy-Lorentz space
coincides with the original one (see Definition \ref{fd2})
in the sense of equivalent quasi-norms.
\end{remark}

The following  Calder\'{o}n reproducing formula is just
\cite[Proposition 2.14]{blyz10}.

\begin{lemma}\label{fivel3}
Let $s\in\mathbb{Z_+}$ and $A$ be a dilation.
Then there exist $\varphi,\,\psi\in\cs(\rn)$ such that
\begin{enumerate}
\item[{\rm(i)}] $\supp\varphi\subset B_0,
\,\int_{\rn}x^\gamma\varphi(x)\,dx=0$ for any
$\gamma\in(\zz_+)^n$ with $|\gamma|\leq s,
\,\widehat{\varphi}(\xi)\geq C$
for any $\xi\in\{x\in\rn:\ a\leq\rho(x)\leq b\}$,
where $0<a<b<1$ and $C$ are positive constants;
\item[{\rm(ii)}] $\supp \widehat{\psi}$
is compact and bounded away from the origin;
\item[{\rm(iii)}] $\sum_{j\in\mathbb{Z}}
\widehat{\psi}((A^\ast)^j\xi)\widehat{\varphi}((A^\ast)^j\xi)=1$
for any $\xi\in\rn\setminus\{\vec{0}_n\}$,
where $A^\ast$ denotes the adjoint matrix of $A$.
\end{enumerate}

Moreover, for any $f\in\cs'_0(\rn),\,f=
\sum_{j\in\mathbb{Z}}f\ast\psi_j\ast\varphi_j$ in $\cs'(\rn)$.
\end{lemma}

Now we prove Theorem \ref{fivet1}.

\begin{proof}[Proof of Theorem \ref{fivet1}]
We first show the necessity of this theorem.
Let $f\in \vhlpq$. It follows from Lemma \ref{fivel1} that
$f\in\cs'_0(\rn)$. On the other hand, for any $k_0\in\zz$,
due to Theorem \ref{ft1} and Remark \ref{fiver1},
we can decompose $f$ as follows
$$f=\sum_{k=-\fz}^{k_0-1}\sum_{i\in\nn}\lik\aik
+\sum_{k=k_0}^{\fz}\sum_{i\in\nn}\cdots=:f_1+f_2,$$
where $\{\lik\}_{i\in\nn,k\in\zz}$ and $\{\aik\}_{i\in\nn,k\in\zz}$
are as in Theorem \ref{ft1}
satisfying \eqref{fe2}. Let $v,u$ be as in Lemma
\ref{fivel2} and $w:=u-v+2\tau$. Then we have
\begin{align}\label{five2}
&\lf\|\chi_{\{x\in\rn:\ S(f)(x)>2^{k_0}\}}\r\|_{\lv}\\
&\hs\ls\lf\|\chi_{\{x\in\rn:\ S(f_1)(x)>2^{k_0-1}\}}\r\|_{\lv}
+\lf\|\chi_{\{x\in E_{k_0}:\ S(f_2)(x)>2^{k_0-1}\}}\r\|_{\lv}\noz\\
&\hs\hs+\lf\|\chi_{\{x\in (E_{k_0})^\com:\ S(f_2)(x)>2^{k_0-1}\}}\r\|_{\lv}\noz\\
&\hs=:{\rm I}_1+{\rm I}_2+{\rm I}_3,\noz
\end{align}
where $E_{k_0}:=\bigcup_{k=k_0}^\fz\bigcup_{i\in\nn}A^{w}\Qik$
and $\{\Qik\}_{i\in\nn,k\in\zz}\st\mathcal{Q}$ are as in Lemma \ref{fivel2}.

Obviously,
\begin{align}\label{five3}
{\rm I}_1&\ls\lf\|\chi_{\{x\in\rn:\ \sum_{k=-\fz}^{k_0-1}\sum_{i\in\mathbb{N}}
\lambda_i^kS(a_i^k)(x)\chi_{A^{w}Q_i^k}(x)>2^{k_0-2}\}}\r\|_{\lv}\\
&\hs+\lf\|\chi_{\{x\in\rn:\ \sum_{k=-\fz}^{k_0-1}\sum_{i\in\mathbb{N}}
\lambda_i^kS(a_i^k)(x)\chi_{(A^{w}Q_i^k)^\com}(x)>2^{k_0-2}\}}\r\|_{\lv}\noz\\
&=:{\rm I}_{1,1}+{\rm I}_{1,2}.\noz
\end{align}
For ${\rm I}_{1,1}$, by \cite[Theorem 3.2]{blyz10}, Lemma \ref{fl2},
Remark \ref{fr1}(i) and a proof similar to that of \eqref{fe8}, we conclude
that
\begin{align}\label{five4}
\lf[\sum_{k_0\in\zz}2^{k_0q}({\rm I}_{1,1})^q\r]^{\frac1q}
\ls\lf[\sum_{k\in\zz}2^{kq}
\lf\|\sum_{i\in\nn}\chi_{\Qik}\r\|^{q}_{\lv}\r]^{\frac1q}
\sim\|f\|_{\vahlpq}.
\end{align}
To deal with ${\rm I}_{1,2}$,
assume that $a$ is a $(p(\cdot),r,s)$-atom
supported on a dyadic cube $Q$. For any $j\in\nn$,
let
\begin{align*}
U_j:=x_{Q}+\lf(B_{v[\ell(Q)-j-1]+u+2\tau}
\setminus B_{v[\ell(Q)-j]+u+2\tau}\r).
\end{align*}
Then, by Lemma \ref{fivel2}(iv),
we know that, for any $x\in(A^{w}Q)^\com$, there exists some $j_0\in\nn$
such that $x\in U_{j_0}$. For this $j_0$,
choose $N\in\nn$ lager enough such that
$$(N-\bz)vj_0+(1-v)\lf(\frac1r-\bz\r)\ell(Q)<0$$
with $\bz$ as in \eqref{fe27}.
By this and an argument similar to that
used in the proof \cite[(3.3)]{lyy16LP}, we find that
\begin{align*}
S(a)(x)\ls b^{Nvj_0}
b^{-\frac{v\ell(Q)}r}\|a\|_{L^r(Q)}.
\end{align*}
From this and
the size condition of $a$, we deduce that
\begin{align}\label{five6}
S(a)(x)&\ls b^{(N-\bz)vj_0+(1-v)(\frac1r-\bz)\ell(Q)}
\lf\|\chi_Q\r\|_{\lv}^{-1}\frac{|Q|^\bz}{b^{[\ell(Q)-j_0]v\bz}}\\
&\ls\lf\|\chi_Q\r\|_{\lv}^{-1}\lf[\frac{|Q|}{\rho(x-x_Q)}\r]^\bz
\ls\lf\|\chi_Q\r\|_{\lv}^{-1}\lf[\HL(\chi_Q)(x)\r]^\bz.\noz
\end{align}
By \eqref{five6}, similarly to \eqref{fe10}, we conclude that
\begin{align*}
\lf[\sum_{k_0\in\zz}2^{k_0q}({\rm I}_{1,2})^q\r]^{\frac1q}
\ls\lf[\sum_{k\in\zz}2^{kq}
\lf\|\sum_{i\in\nn}\chi_{\Qik}\r\|^{q}_{\lv}\r]^{\frac1q}
\sim\|f\|_{\vahlpq},
\end{align*}
which, together with \eqref{five3} and \eqref{five4},
further implies that
\begin{align}\label{five7}
\lf[\sum_{k_0\in\zz}2^{k_0q}({\rm I}_1)^q\r]^{\frac1q}
\ls\|f\|_{\vahlpq}.
\end{align}

For ${\rm I}_2$ and ${\rm I}_3$, from \eqref{five6}
and a proof similar to those of \eqref{fe11} and \eqref{fe28},
it follows that
\begin{align}\label{five8}
\lf[\sum_{k_0\in\zz}2^{k_0q}({\rm I}_2)^q\r]^{\frac1q}
\ls\|f\|_{\vahlpq}\hspace{0.3cm}{\rm and}
\hspace{0.3cm}\lf[\sum_{k_0\in\zz}2^{k_0q}({\rm I}_3)^q\r]^{\frac1q}
\ls\|f\|_{\vahlpq}.
\end{align}

Combining \eqref{five2}, \eqref{five7} and \eqref{five8},
we obtain
\begin{align*}
\|S(f)\|_{\vlpq}&\sim\lf[\sum_{k_0\in\zz}2^{k_0q}
\lf\|\chi_{\{x\in\rn:\ |S(f)(x)|>2^{k_0}\}}\r\|_{\lv}^q\r]^{\frac1q}\\
&\ls\lf[\sum_{k_0\in\zz}2^{k_0q}({\rm I}_1)^q\r]^{\frac1q}+
\lf[\sum_{k_0\in\zz}2^{k_0q}({\rm I}_2)^q\r]^{\frac1q}+
\lf[\sum_{k_0\in\zz}2^{k_0q}({\rm I}_3)^q\r]^{\frac1q}\\
&\ls\|f\|_{\vahlpq}\sim\|f\|_{\vhlpq}
\end{align*}
with the usual modification made when $q=\fz$,
which implies that $S(f)\in\vlpq$ and
$$\|S(f)\|_{\vlpq}\ls\|f\|_{\vhlpq}.$$
This shows the necessity of Theorem \ref{fivet1}.

Next we prove the sufficiency of Theorem \ref{fivet1}.
Let $f\in\cs'_0(\rn)$ and $S(f)\in\vlpq$. Then we need
to prove that $f\in\vhlpq$ and
\begin{align}\label{five19}
\|f\|_{\vhlpq}\ls\|S(f)\|_{\vlpq}.
\end{align}
To this end, for any $k\in\mathbb{Z}$, let
$\Omega_k:=\{x\in\rn:\ S(f)(x)>2^k\}$ and
$$\mathcal{Q}_k:=\lf\{Q\in\mathcal{Q}:
\ |Q\cap\Omega_k|>\frac{|Q|}2\ \ {\rm and}\
\ |Q\cap\Omega_{k+1}|\leq\frac{|Q|}2\r\}.$$
Clearly, for each $Q\in\mathcal{Q}$,
there exists a unique $k\in\mathbb{Z}$
such that $Q\in\mathcal{Q}_k$.
Let $\{Q_i^k\}_i$ be the set of all \emph{maximal dyadic cubes}
in $\mathcal{Q}_k$,
namely, there exists no $Q\in\mathcal{Q}_k$
such that $Q_i^k\subsetneqq Q$ for any $i$.

Let $u$, $v$ be as in Lemma \ref{fivel2} and,
for any $Q\in\mathcal{Q}$,
$$\widehat{Q}:=\lf\{(y,t)\in\rn\times\mathbb{R}:\
y\in Q,\ v\ell(Q)+u+\tau\leq t<v[\ell(Q)-1]+u+\tau\r\}.$$
Then $\{\widehat{Q}\}_{Q\in\mathcal{Q}}$ are mutually disjoint and
\begin{align}\label{five9}
\rn\times\mathbb{R}=\bigcup_{k\in\zz}\bigcup_i\
\bigcup_{Q\subset Q_i^k,\,Q\in\mathcal{Q}_k}\widehat{Q}
=:\bigcup_{k\in\zz}\bigcup_i B_{k,\,i},
\end{align}
Obviously, $\{B_{k,i}\}_{k\in\zz,\,i}$
are mutually disjoint by Lemma \ref{fivel2}(ii).

Let $\psi$ and $\varphi$ be as in Lemma \ref{fivel3}. Then, by
Lemma \ref{fivel3}, the properties of the tempered distributions
(see \cite[Theorem 2.3.20]{lg08} or \cite[Theorem 3.13]{sw71})
and \eqref{five9}, we find that, for any $f\in\cs'_0(\rn)$ with
$S(f)\in\vlpq$ and $x\in\rn$,
\begin{align*}
f(x)
&=\sum_{k\in\mathbb{Z}}f\ast\psi_k\ast\varphi_k(x)
=\int_{\rn\times\mathbb{R}}
f\ast\psi_t(y)\ast\varphi_t(x-y)\,dydm(t)\\
&=\sum_{k\in\mathbb{Z}}\sum_i\int_{B_{k,\,i}}
f\ast\psi_t(y)\ast\varphi_t(x-y)\,dydm(t)
=:\sum_{k\in\mathbb{Z}}\sum_i h_i^k(x)
\end{align*}
in $\cs'(\rn)$, where $m(t)$ is
the \emph{counting measure} on $\mathbb{R}$, namely,
for any set $E\st\mathbb{R}$, $m(E):=\sharp E$ if $E$ has only finite elements,
or else $m(E):=\fz$. Moreover, by
an argument similar to that used in the proofs of
\cite[(3.24), (3.29) and (3.30)]{lyy16LP}, it is easy to see that
there exists some $C_0\in(0,\fz)$ such that,
for any $r\in (\max\{p_+,1\},\fz)$, $k\in\zz$, $i$ and
$\az\in(\zz_+)^n$ as in Definition \ref{fd1},
\begin{align}\label{five10}
\supp h_i^k
\subset x_{Q_i^k}+B_{v[\ell(Q_i^k)-1]+
u+3\tau}=:B_i^k,
\end{align}
\begin{align}\label{five11}
\lf\|h_i^k\r\|_{L^r(\rn)}
\le C_02^k\lf|B_i^k\r|^{1/r}
\end{align}
and
\begin{align}\label{five12}
\int_{\rn}h_i^k(x)x^\az\,dx=0,
\end{align}
namely, $h_i^k$ is a $(p(\cdot),r,s)$-atom multiplied by a constant.

For any $k\in\zz$ and $i$,
let $\lik:=C_02^k\|\chi_{B_i^k}\|_\lv$ and
$a_i^k:=(\lambda_i^k)^{-1}h_i^k$, where
$C_0$ is a positive constant as in \eqref{five11}.
Then we obtain
$$f=\sum_{k\in\mathbb{Z}}\sum_i h_i^k
=\sum_{k\in\zz}\sum_i
\lambda_i^k a_i^k\qquad {\rm in}\quad \cs'(\rn).$$
By \eqref{five10} and \eqref{five12},
we know that $\supp a_i^k\subset B_i^k$ and
$a_i^k$ also has the vanishing moments up to $s$.
From \eqref{five11} and Lemma \ref{fivel2}(iv),
it follows that
$\|a_i^k\|_{L^r(\rn)}\leq\|\chi_{B_i^k}\|_\lv^{-1}|B_i^k|^{1/r}$.
Therefore, $a_i^k$ is a $(p(\cdot),r,s)$-atom for any $k\in\zz$ and $i$.
Moreover, by Theorem \ref{ft1},
the mutual disjointness of $\{Q_i^k\}_{k\in\mathbb{Z},\,i}$,
Lemma \ref{fivel2}(iv) again, $|\Qik\cap\Omega_k|\ge\frac{|\Qik|}{2}$,
Lemmas \ref{fivel4} and \ref{sl3*},
we conclude that
\begin{align*}
\|f\|_{\vhlpq}
&\sim\lf[\sum_{k\in\zz}\lf\|\lf\{\sum_{i\in\nn}
\lf[\frac{\lz_i^k\chi_{\Bik}}{\|\chi_{\Bik}\|_{\lv}}\r]^
{\underline{p}}\r\}^{1/\underline{p}}\r\|_{\lv}^q\r]^{\frac1q}\\
&\sim\lf[\sum_{k\in\zz}2^{kq}\lf\|\lf(\sum_{i\in\nn}
\chi_{\Bik}\r)^{\frac 1{p_-}}\r\|_{\lv}^q\r]^{\frac1q}
\sim\lf[\sum_{k\in\zz}2^{kq}\lf\|\lf(\sum_{i\in\nn}
\chi_{\Qik}\r)^{\frac 1{p_-}}\r\|_{\lv}^q\r]^{\frac1q}\\
&\sim\lf[\sum_{k\in\zz}2^{kq}\lf\|\lf(\sum_{i\in\nn}
\chi_{\Qik}\r)^{1/2}\r\|
_{L^{\frac{2p(\cdot)}{p_-}}(\rn)}^{\frac {2q}{p_-}}\r]^{\frac1q}
\ls\lf[\sum_{k\in\zz}2^{kq}\lf\|\lf(\sum_{i\in\nn}
\chi_{\Qik\cap\Omega_k}\r)^{1/2}\r\|
_{L^{\frac{2p(\cdot)}{p_-}}(\rn)}^{\frac {2q}{p_-}}\r]^{\frac1q}\\
&\sim\lf(\sum_{k\in\zz}2^{kq}\lf\|\chi_{\Omega_k}\r\|_{\lv}^q\r)^{\frac1q}
\sim\|S(f)\|_{L^{p(\cdot),q}(\rn)}
\end{align*}
with the usual modification made when $q=\fz$,
which implies that $f\in\vhlpq$ and
$$\|f\|_{\vhlpq}\ls\|S(f)\|_{\vlpq}.$$
This finishes the proof of \eqref{five19} and hence
Theorem \ref{fivet1}.
\end{proof}

\section{Real interpolation \label{s6}}
\hskip\parindent
As another application of the atomic characterization of $\vhlpq$,
in this section, we obtain a real interpolation result between
$\vh$ and $\lfz$. Moreover, using this result, together with
\cite[Corollary 4.20]{zsy16} and \cite[Remark 4.2(ii)]{kv14},
we then show
that the anisotropic variable Hardy-Lorentz space
$\vhlpq$ with $p_-\in(1,\infty)$
coincides with the variable Lorentz space $\vlpq$.

To state the main result of this section,
we first recall some basic notions about the real interpolation
(see \cite{bl76}). Assume that $(X_1,\,X_2)$ is a
compatible couple of quasi-normed spaces, namely,
$X_1$ and $X_2$ are two quasi-normed linear spaces
which are continuously embedded in some larger
topological vector space. Let
$$X_1+X_2:=\{f_1+f_2:\ f_1\in X_1,\,f_2\in X_2\}.$$
For any $t\in(0,\fz]$, the \emph{Peetre $K$-functional} on $X_1+X_2$ is defined
by setting, for any $f\in X_0+X_1$,
$$K(t,f;X_1,X_2):=\inf\lf\{\|f_1\|_{X_1}+t\|f_2\|_{X_2}:\
f=f_1+f_2,\,f_1\in X_1\ {\rm and}\ f_2\in X_2\r\}.$$
Moreover, for any $\theta\in(0,1)$ and $q\in(0,\fz]$,
the \emph{real interpolation space} $(X_1,\,X_2)_{\theta,q}$
is defined as
\begin{align*}
(X_1,\,X_2)_{\theta,q}:=\lf\{f\in X_1+X_2:\
\|f\|_{\theta,q}:=\lf[\int_0^\fz\lf\{t^{-\theta}
K(t,f;X_1,X_2)\r\}^q\,\frac{dt}t\r]^{1/q}<\fz\r\}.
\end{align*}

\begin{definition}\label{sixd1}
Let $p(\cdot)\in C^{\log}(\rn)$ and
$N\in\mathbb{N}\cap[\lfloor(\frac1{\underline{p}}-1)\frac{\ln b}{\ln
\lambda_-}\rfloor+2,\fz)$,
where $\underline{p}$ is as in \eqref{se4}.
The \emph{anisotropic variable Hardy space},
denoted by $\vh$, is defined by setting
\begin{equation*}
\vh
:=\lf\{f\in\cs'(\rn):\ M_N^0(f)\in\lv\r\}
\end{equation*}
and, for any $f\in\vh$, let
$\|f\|_{\vh}:=\| M_N^0(f)\|_{\lv}$.
\end{definition}

The main result of this section is stated as follows.

\begin{theorem}\label{sixt1}
Let $p(\cdot)\in C^{\log}(\rn)$, $q\in(0,\fz]$ and $\theta\in(0,1)$.
Then it holds true that
\begin{equation*}
(\vh, L^\fz(\rn))_{\theta,q}=H_A^{\wz p(\cdot),q}(\rn),
\end{equation*}
where $\frac1{\wz p(\cdot)}=\frac{1-\theta}{p(\cdot)}$.
\end{theorem}

As a consequence of Theorem \ref{sixt1}, \cite[Corollary 4.20]{zsy16}
and \cite[Remark 4.2(ii)]{kv14},
we immediately obtain the following conclusion, the details being omitted.

\begin{corollary}\label{sixc1}
Let $p(\cdot)\in C^{\log}(\rn)$. If $p_-\in(1,\fz)$ and $q\in(0,\fz]$, then
$\vhlpq=\vlpq$ with equivalent quasi-norm.
\end{corollary}

\begin{remark}\label{sixr1}
\begin{enumerate}
\item[(i)] When $p(\cdot)\equiv p\in(0,1]$, Theorem \ref{sixt1} goes back
to \cite[Lemma 6.3]{lyy16}, which states that, for any $\theta\in(0,1)$
and $q\in(0,\fz]$,
$$(H_A^{p}(\rn),L^\fz(\rn))_{\theta,q}=H_A^{p/(1-\theta),q}(\rn).$$

\item[(ii)] When $p(\cdot)\equiv p\in(1,\fz)$, Theorem \ref{sixt1} coincides
with \cite[Remark 6.7]{lyy16} (see also \cite[Theorem 7]{rs73}), namely,
for any $\theta\in(0,1)$ and $q\in(0,\fz]$,
$$(L^p(\rn),L^\fz(\rn))_{\theta,q}=L^{p/(1-\theta),q}(\rn).$$

\item[(iii)] Let $A:=d\,{\rm I}_{n\times n}$ for some $d\in\rr$
with $|d|\in(1,\fz)$. Then $H^{p}_A(\rn)$ and $H_A^{p/(1-\theta),q}(\rn)$
in (i) of this remark become the classical isotropic Hardy and
Hardy-Lorentz spaces, respectively. In this case, the result in (i) of this remark is just
\cite[Theorem 1]{frs74}. In addition, $\vh$ and $H_A^{\wz p(\cdot),q}(\rn)$
in Theorem \ref{sixt1} become  the classical isotropic variable Hardy
and Hardy-Lorentz spaces, respectively. In this case,
Theorem \ref{sixt1} includes the result in \cite[Theorem 1.5]{zyy17}
as a special case and Theorem \ref{sixt1} with $p_-\in(1,\fz)$
coincides with \cite[Remark 4.2(ii)]{kv14}.
\end{enumerate}
\end{remark}

To prove Theorem \ref{sixt1}, we need the following technical lemma,
which decomposes any distribution $f$ from the anisotropic variable
Hardy-Lorentz space $f\in H_A^{\wz p(\cdot),q}(\rn)$ into ``good" and ``bad" parts
and plays a key role in the proof of Theorem \ref{sixt1}.

\begin{lemma}\label{sixl1}
Let $\theta\in(0,1)$, $q\in(0,\fz]$, $p(\cdot)$ and $\wz p(\cdot)$
be as in Theorem \ref{sixt1}.
Then, for any $f\in H_A^{\wz p(\cdot),q}(\rn)$ and $k\in\zz$,
there exist $g_k\in L^\fz(\rn)$
and $b_k\in\cs'(\rn)$ such that $f=g_k+b_k$ in $\cs'(\rn)$,
$\|g_k\|_{L^\fz(\rn)}\le C2^k$
and
\begin{equation}\label{sixe2}
\lf\|b_k\r\|_{H_A^{p(\cdot)}(\rn)}
\le \widetilde{C}\lf\|M_N(f)\chi_{\{x\in\rn:\ M_N(f)(x)>2^k\}}\r\|_{\lv}<\fz,
\end{equation}
where, for any
$N\in\mathbb{N}\cap[\lfloor(\frac1{\underline{p}}-1)
\frac{\ln b}{\ln\lambda_-}\rfloor+2,\fz)$
with $\underline{p}$ as in \eqref{se4},
$M_N(f)$ is as in \eqref{se8},
$C$ and $\widetilde{C}$ are two positive constants independent of $f$ and $k$.
\end{lemma}

\begin{proof}
Let all the notation be the same as those used in the proof of Theorem \ref{ft1}.
For any $f\in H_A^{\wz p(\cdot),q}(\rn)$, $k\in\mathbb{Z}$ and
$N\in\mathbb{N}\cap[\lfloor(\frac1{\underline{p}}-1)\frac{\ln b}{\ln
\lambda_-}\rfloor+2,\fz)$,
let
$$\Omega_k:=\lf\{x\in\rn:\ M_N(f)(x)>2^k\r\}.$$
Then, for any $k_0\in\zz$, by an argument similar
to that used in the proof of \eqref{fe13}, we have
\begin{align*}
f=\sum_{k\in\zz}\sum_{i\in\nn}h_i^k
=\sum_{k=-\fz}^{k_0}\sum_{i\in\nn}h_i^k
+\sum_{k=k_0+1}^\fz\sum_{i\in\nn}\cdots=:g_{k_0}+b_{k_0}
\qquad{\rm in}\ \cs'(\rn),
\end{align*}
where, for any $k\in\zz$ and $i\in\nn$, $h_i^k$ is a
$(p(\cdot),\fz,s)$-atom multiplied by a constant  and satisfies that
\begin{align}\label{sixe3}
\supp h_i^k\subset (x_i^k+B_{\ell_i^k+4\tau}),
\end{align}
\begin{align}\label{sixe4}
\lf\|h_i^k\r\|_{L^{\infty}(\rn)}\ls2^k
\end{align}
and
\begin{align}\label{sixe5}
\int_\rn h_i^k(x)Q(x)dx=0
\qquad {\rm for\ any}\ Q\in\mathcal{P}_{m}(\rn).
\end{align}

By the finite intersection property of
$\{x_i^k+B_{\ell_i^k+4\tau}\}_{i\in\nn}$
for each $k\in\zz$ (see \eqref{fe18}), \eqref{sixe3}
and \eqref{sixe4}, we conclude that, for any $k_0\in\zz$,
\begin{align*}
\lf\|g_{k_0}\r\|_{\lfz}
\ls\sum_{k=-\fz}^{k_0}2^k\sim2^{k_0}.
\end{align*}

On the other hand, for any $k\in\zz$ and $i\in\nn$, let
$a_i^k:=(\lambda_i^k)^{-1}h_i^k$, where
$$\lambda_i^k\sim 2^k\lf\|\chi_{\xik+B_{\ell_i^k+4\tau}}\r\|_{\lv}.$$
Then, by this, \eqref{sixe3}, \eqref{sixe4} and \eqref{sixe5},
we know that, for any $k\in\zz$ and $i\in\nn$, $a_i^k$ is a
$(p(\cdot),\fz,s)$-atom. Therefore, by the finite intersection property of
$\{x_i^k+B_{\ell_i^k+4\tau}\}_{i\in\nn}$
for each $k\in\zz$ again and the fact that
$\Omega_k=\bigcup_{i\in\mathbb{N}}(x_i^k+B_{\ell_i^k})$
(see \eqref{fe16}), we find that
\begin{align}\label{sixe6}
&\lf\|\lf\{\sum_{k=k_0+1}^\fz\sum_{i\in\nn}
\lf[\frac{\lik\chi_{x_i^k+B_{\ell_i^k+4\tau}}}
{\|\chi_{x_i^k+B_{\ell_i^k+4\tau}}\|_\lv}\r]^{\underline{p}}
\r\}^{\frac1{\underline{p}}}\r\|_{\lv}\\
&\hs\ls \lf\|\lf(\sum_{k=k_0+1}^\fz\sum_{i\in\nn}2^{k\underline{p}}
\chi_{x_i^k+B_{\ell_i^k+4\tau}}\r)^{\frac1{\underline{p}}}\r\|_{\lv}\noz\\
&\hs\ls \lf\|\lf(\sum_{k=k_0+1}^\fz2^{k\underline{p}}
\chi_{\Omega_k}\r)^{\frac1{\underline{p}}}\r\|_{\lv}
\sim \lf\|\lf[\sum_{k=k_0+1}^\fz\lf(2^{k}
\chi_{\Omega_k\backslash \Omega_{k+1}}\r)^{\underline{p}}
\r]^{\frac1{\underline{p}}}\r\|_{\lv}\noz\\
&\hs\ls \lf\|M_{N}(f)\lf(\sum_{k=k_0+1}^\fz
\chi_{\Omega_k\backslash \Omega_{k+1}}
\r)^{\frac1{\underline{p}}}\r\|_{\lv}
\sim \lf\|M_{N}(f)\chi_{\Omega_{k_0+1}}\r\|_{\lv}\noz\\
&\hs\ls \lf\|M_{N}(f)\chi_{\{x\in\rn:\ M_{N}(f)(x)>2^{k_0}\}}\r\|_{\lv}.\noz
\end{align}
Moreover, from Remark \ref{sr1}(i) and Lemma \ref{sl3*},
it follows that
\begin{align}\label{sixe7}
&\lf\|M_{N}(f)\chi_{\{x\in\rn:\ M_{N}(f)(x)>2^{k_0}\}}\r\|_{\lv}\\
&\hs\le\lf\{\sum_{j\in\zz_+}\lf\|\lf[M_N(f)\r]^{\underline{p}}
\chi_{\{x\in\rn:\ 2^{j+k_0}<M_N(f)(x)\le 2^{j+k_0+1}\}}\r\|
_{L^{\frac{p(\cdot)}{\underline{p}}}(\rn)}\r\}
^{\frac1{\underline{p}}}\noz\\
&\hs\ls \lf[\sum_{j\in\zz_+}(2^{j+k_0})^{\underline{p}}
\lf\|\chi_{\{x\in\rn:\ M_N(f)(x)>2^{j+k_0}\}}\r\|_{\lv}^{\underline{p}}
\r]^{\frac1{\underline{p}}}\noz\\
&\hs\sim \lf\{\sum_{j\in\zz_+}(2^{j+k_0})^{-\theta\underline{p}/(1-\theta)}
\lf[2^{j+k_0}\lf\|\chi_{\{x\in\rn:\ M_N(f)(x)>2^{j+k_0}\}}
\r\|_{L^{\wz p(\cdot)}(\rn)}\r]^{\underline{p}/(1-\theta)}
\r\}^{\frac1{\underline{p}}}\noz\\
&\hs\ls \lf\|M_N(f)\r\|_{L^{\wz p(\cdot),\fz}(\rn)}^{\frac1{1-\theta}}
\lf[\sum_{j\in\zz_+}\lf(2^{j+k_0}\r)^{-\theta\underline{p}/(1-\theta)}\r]
^{\frac1{\underline{p}}}\noz\\
&\hs\sim \lf(2^{k_0}\r)^{-\theta/(1-\theta)}
\|f\|_{H_A^{\wz p(\cdot),\fz}(\rn)}^{\frac1{1-\theta}}
\ls\lf(2^{k_0}\r)^{-\theta/(1-\theta)}
\|f\|_{H_A^{\wz p(\cdot),q}(\rn)}^{\frac1{1-\theta}}<\fz.\noz
\end{align}
Observe that $(\rn,\,\rho,\,dx)$ is an RD-space (see \cite{hmy08,zy11}).
From this,
\cite[Theorem 4.3(i)]{zsy16}, \eqref{sixe6} and \eqref{sixe7},
we further deduce that
\begin{align*}
\lf\|b_{k_0}\r\|_{\vh}
&\ls\lf\|\lf\{\sum_{k=k_0+1}^\fz\sum_{i\in\nn}
\lf[\frac{\lik\chi_{x_i^k+B_{\ell_i^k+4\tau}}}
{\|\chi_{x_i^k+B_{\ell_i^k+4\tau}}\|_\lv}\r]^{\underline{p}}
\r\}^{\frac1{\underline{p}}}\r\|_{\lv}\\
&\ls\lf\|M_{N}(f)\chi_{\{x\in\rn:\ M_{N}(f)(x)>2^{k_0}\}}\r\|_{\lv}.
\end{align*}
This finishes the proof of Lemma \ref{sixl1}.
\end{proof}

Now we prove Theorem \ref{sixt1}.

\begin{proof}[Proof of Theorem \ref{sixt1}]
We first prove that
\begin{equation}\label{sixe8}
H_A^{\wz p(\cdot),q}(\rn)\st (\vh,L^\fz(\rn))_{\theta,q}.
\end{equation}

To this end, let $f\in H_A^{\wz p(\cdot)q}(\rn)$. Then,
by Lemma \ref{sixl1}, we know that, for any $k\in\zz$,
there exist $g_k\in L^\fz(\rn)$ and $b_k\in \vh$
such that $f= g_k+b_k$ in $\cs'(\rn)$, $\|g_k\|_{L^\fz(\rn)}\ls2^k$ and
$b_k$ satisfies \eqref{sixe2}. By this and a proof similar to that of
\cite[(3.3)]{zyy17}, we find that, for any $t\in(0,\fz)$,
\begin{align}\label{sixe9}
K(t,f;\vh,L^\fz(\rn))\ls 2^{k(t)}t,
\end{align}
where, for any $t\in(0,\fz)$,
$$k(t):=\inf\lf\{\ell\in\zz:\ \lf[\sum_{j\in\zz_+}
\lf\{2^j h(2^{j+\ell})\r\}^{\underline{p}}
\r]^{\frac1{\underline{p}}}\le t\r\}$$
and, for any $\lz\in(0,\fz)$,
$$h(\lz):=\lf\|\chi_{\{x\in\rn:\ M_N(f)(x)>\lz\}}\r\|_{\lv}.$$

On the other hand, by \eqref{sixe9},
it is easy to see that, for any $\theta\in(0,1)$ and $q\in(0,\fz]$,
\begin{align}\label{sixe10}
&\lf\{\int_0^\fz t^{-\theta q}
\lf[K(t,f;\vh,L^\fz(\rn))\r]^q\,\frac{dt}t\r\}^{1/q}\\
&\hs\ls\lf[\sum_{\ell\in\zz}2^{\ell q}\int_{\{t\in(0,\fz):\
2^\ell<2^{k(t)}\le2^{\ell+1}\}}t^{(1-\theta)q}\,\frac{dt}t\r]^{1/q}
\ls\lf[\sum_{\ell\in\zz}2^{\ell q}\int_0^{\{\sum_{j\in\zz_+}
[2^j h(2^{j+\ell})]^{\underline{p}}\}^{\frac1{\underline{p}}}}
t^{(1-\theta)q}\,\frac{dt}t\r]^{1/q}\noz\\
&\hs\sim\lf[\sum_{\ell\in\zz}2^{\ell q}\lf\{\sum_{j\in\zz_+}
\lf[2^j h(2^{j+\ell})\r]^{\underline{p}}\r\}
^{\frac{(1-\theta)q}{\underline{p}}}\r]^{1/q}=:C{\rm J},\noz
\end{align}
where $C$ is a positive constant independent of $f$. Next we estimate
${\rm I}$ by considering two cases.

\emph{Case 1.} $\frac{(1-\theta)q}{\underline{p}}\in(0,1]$. For this case,
by the well-known inequality that,
for any $d\in(0,1]$ and $\{a_i\}_{i\in\nn}\st\cc$,
$$\lf(\sum_{i\in\nn}|a_i|\r)^d\le \sum_{i\in\nn}|a_i|^d,$$
we obtain
\begin{align}\label{sixe11}
{\rm J}^q&\le\sum_{\ell\in\zz}2^{\ell q}\sum_{j\in\zz_+}2^{j(1-\theta)q}
\lf[h(2^{j+\ell})\r]^{(1-\theta)q}
=\sum_{\ell\in\zz}\sum_{j\in\zz_+}2^{(\ell-j)q}2^{j(1-\theta)q}
\lf[h(2^{\ell})\r]^{(1-\theta)q}\\
&=\sum_{\ell\in\zz}2^{\ell q}\lf[h(2^{\ell})\r]^{(1-\theta)q}
\sum_{j\in\zz_+}2^{-j\theta q}
\sim\sum_{\ell\in\zz}2^{\ell q}\lf[h(2^{\ell})\r]^{(1-\theta)q}.\noz
\end{align}

\emph{Case 2.} $\frac{(1-\theta)q}{\underline{p}}\in(1,\fz]$. For this case,
let $\eta:=\frac{(1-\theta)q}{\underline{p}}$. Then, from the H\"older inequality,
it follows that, for any $\dz\in(0,\frac{\theta}{1-\theta})$,
\begin{align}\label{sixe12}
{\rm J}&\le\lf\{\sum_{\ell\in\zz}2^{\ell q}\lf(\sum_{j\in\zz_+}
2^{-\dz\underline{p}j\eta'}\r)^{\frac{\eta}{\eta'}}
\sum_{j\in\zz_+}2^{(1+\dz)\underline{p}j\eta}\lf[
h(2^{j+\ell})\r]^{\underline{p}\eta}\r\}^{1/q}\\
&\ls\lf\{\sum_{\ell\in\zz}2^{\ell q}
\sum_{j\in\zz_+}2^{(1+\dz)(1-\theta)jq}\lf[
h(2^{j+\ell})\r]^{(1-\theta)q}\r\}^{1/q}\noz\\
&\sim\lf\{\sum_{\ell\in\zz}2^{\ell q}\lf[
h(2^{j+\ell})\r]^{(1-\theta)q}
\sum_{j\in\zz_+}2^{[\dz(1-\theta)-\theta]jq}\r\}^{1/q}\sim\lf\{\sum_{\ell\in\zz}2^{\ell q}
\lf[h(2^{\ell})\r]^{(1-\theta)q}\r\}^{1/q}.\noz
\end{align}

Combining \eqref{sixe10}, \eqref{sixe11} and \eqref{sixe12},
we conclude that
\begin{align*}
&\lf\{\int_0^\fz t^{-\theta q}
\lf[K(t,f;\vh,L^\fz(\rn))\r]^q\,\frac{dt}t\r\}^{1/q}\\
&\hs\ls\lf[\sum_{\ell\in\zz}2^{\ell q}
\lf\|\chi_{\{x\in\rn:\ M_N(f)(x)>2^\ell\}}\r\|
_{\lv}^{(1-\theta)q}\r]^{1/q}\sim\lf[\sum_{\ell\in\zz}2^{\ell q}
\lf\|\chi_{\{x\in\rn:\ M_N(f)(x)>2^\ell\}}\r\|
_{L^{\wz p(\cdot)}(\rn)}^{q}\r]^{1/q}\\
&\hs\sim\lf\|M_N(f)\r\|_{L^{\wz p(\cdot),q}(\rn)}
\sim\|f\|_{H_A^{\wz p(\cdot),q}(\rn)}
\end{align*}
with the usual modification made when $q=\fz$,
which implies that $$f\in(\vh,L^\fz(\rn))_{\theta,q}$$ and
hence completes the proof of \eqref{sixe8}.

Conversely, we need to show that
\begin{equation}\label{sixe13}
(\vh,L^\fz(\rn))_{\theta,q}\st H_A^{\wz p(\cdot),q}(\rn).
\end{equation}
To this end, we claim that $M_N$ is bounded from the space
$(\vh,L^\fz(\rn))_{\theta,q}$ to
the space $(\lv,L^\fz(\rn))_{\theta,q}$.
Indeed, let $f\in (\vh,L^\fz(\rn))_{\theta,q}$. Then, by definition,
we know that there exist
$f_1\in \vh$ and $f_2\in L^\fz(\rn)$ such that
\begin{equation}\label{sixe14}
\lf\{\int_0^\fz t^{-\theta q}
\lf[\|f_1\|_{\vh}+t\|f_2\|_{L^\fz(\rn)}\r]^q\,\frac{dt}t\r\}^{1/q}
\ls \|f\|_{(\vh,L^\fz(\rn))_{\theta,q}}.
\end{equation}
On the other hand, since $M_N$ is bounded from
$\vh$ to $\lv$ and also from $L^\fz(\rn)$ to $L^\fz(\rn)$,
it follows that $M_N(f_1)\in \lv$ and $M_N(f_2)\in L^\fz(\rn)$.
For any $i\in\{1,2\}$, let
$$E_i:=\lf\{x\in\rn:\ \frac 12 M_N(f)(x)\le M_N(f_i)(x)\r\}.$$
Then $\rn=E_1\cup E_2$. Thus, we have
$$M_N(f)=M_N(f)\chi_{E_1}+M_N(f)
\chi_{E_2\backslash E_1}\in \lv+L^\fz(\rn),$$
which, combined with \eqref{sixe14}, further implies that
\begin{align*}
&\lf\|M_N(f)\r\|_{(\lv,L^\fz(\rn))_{\theta,q}}\\
&\hs\le\lf\{\int_0^\fz t^{-\theta q}\lf[\lf\|M_N(f)
\chi_{E_1}\r\|_{\lv}+t\lf\|M_N(f)\chi_{E_2\backslash E_1}
\r\|_{\lfz}\r]^q\,\frac{dt}t\r\}^{1/q}\\
&\hs\ls\lf\{\int_0^\fz t^{-\theta q}\lf[\lf\|M_N(f_1)
\r\|_{\lv}+t\lf\|M_N(f_2)\r\|_{\lfz}\r]^q\,\frac{dt}t\r\}^{1/q}\\
&\hs\ls\lf\{\int_0^\fz t^{-\theta q}\lf[\|f_1\|_{\vh}
+t\|f_2\|_{\lfz}\r]^q\,\frac{dt}t\r\}^{1/q}
\ls\|f\|_{(\vh,L^\fz(\rn))_{\theta,q}}
\end{align*}
with the usual modification made when $q=\fz$.
Therefore, the above claim holds true.

By this claim and \cite[Remark 4.2(ii)]{kv14}, we find that,
if $f\in (\vh,L^\fz(\rn))_{\theta,q}$, then
$M_N(f)$ belongs to $L^{\wz p(\cdot),q}(\rn)$, namely,
$f\in H_A^{\wz p(\cdot),q}(\rn)$. Thus, \eqref{sixe13} holds true.
This finishes the proof of Theorem \ref{sixt1}.
\end{proof}

\bigskip

\noindent  Jun Liu, Dachun Yang (Corresponding author) and Wen Yuan

\medskip

\noindent  School of Mathematical Sciences, Beijing Normal University,
Laboratory of Mathematics and Complex Systems, Ministry of
Education, Beijing 100875, People's Republic of China

\smallskip

\noindent {\it E-mails}: \texttt{junliu@mail.bnu.edu.cn} (J. Liu)

\hspace{1cm}\texttt{dcyang@bnu.edu.cn} (D. Yang)

\hspace{1cm}\texttt{wenyuan@bnu.edu.cn} (W. Yuan)

\end{document}